\documentclass[hidelinks,onefignum,onetabnum]{siamart220329}

\usepackage{amsfonts}
\usepackage{booktabs}

\ifpdf
  \DeclareGraphicsExtensions{.eps,.pdf,.png,.jpg}
\else
  \DeclareGraphicsExtensions{.eps}
\fi

\newsiamremark{remark}{Remark}
\newsiamremark{hypothesis}{Hypothesis}
\crefname{hypothesis}{Hypothesis}{Hypotheses}
\newsiamthm{claim}{Claim}
\newsiamremark{example}{Example}

\headers{An optimal transport analogue of the ROF model}{T. Milne, A. Nachman}

\title{An optimal transport analogue of the Rudin Osher Fatemi model and its corresponding multiscale theory\thanks{This work appeared in part in the first author's Ph.D.~thesis. 
\funding{This work was funded by NSERC Discovery Grant RGPIN-06329.}}}

\author{Tristan Milne\thanks{Department of Mathematics, University of Toronto, 40 St. George Street, Toronto Ontario Canada
  (\email{tmilne@math.toronto.edu}, \email{nachman@math.toronto.edu}).}
\and Adrian Nachman\footnotemark[2] \thanks{
 The Edward S. Rogers Sr. Department of Electrical and Computer Engineering, University of Toronto, 10 King's College Road, Toronto Ontario Canada}.}

\ifpdf
\hypersetup{
  pdftitle={An optimal transport analogue of the Rudin Osher Fatemi model},
  pdfauthor={T. Milne and A. Nachman}
}
\fi

\usepackage{mathtools}
\usepackage{enumitem}

\newcommand{\RR}{\mathbb{R}}

\newcommand{\step}{\lambda}

\newcommand{\onelip}{1\text{\upshape-Lip}(\Omega)}

\newcommand{\steplip}{\step\text{\upshape-Lip}(\Omega)}

\newcommand{\muc}{\mu^a}

\newcommand{\mud}{\mu^b}
\newcommand{\mudi}{\mu_i^b}
\newcommand{\rhonc}{\rho_\step^a}

\newcommand{\rhond}{\rho_\step^b}

\newcommand{\lebesgue}{\mathcal{L}_d}

\newcommand{\otplan}{\gamma_0}

\newcommand{\hybrid}{c_{2,\step}}

\newcommand{\squaredw}[2]{\frac{1}{2}W_2^2(#1, #2)}
\newcommand{\cI}{\mathcal{I}}
\newcommand{\bregmandiv}[1]{D_{\step_{#1}}}
\newcommand{\otmapnu}{S_\step}

\DeclarePairedDelimiter\norm{\lVert}{\rVert}%
\DeclarePairedDelimiter\abs{\lvert}{\rvert}%

\DeclareMathOperator*{\argmin}{arg\,min}
\DeclareMathOperator*{\essinf}{ess\,inf}
\DeclareMathOperator*{\spt}{spt}

\DeclareMathOperator*{\diam}{diam}

\makeatletter
\let\oldabs\abs
\def\abs{\@ifstar{\oldabs}{\oldabs*}}
\let\oldnorm\norm
\def\norm{\@ifstar{\oldnorm}{\oldnorm*}}

\makeatother

\begin{document}

\maketitle

\begin{abstract}
We develop a theory for image restoration with a learned regularizer that is analogous to that of Meyer’s characterization of solutions of the classical variational method of Rudin-Osher-Fatemi (ROF). The learned regularizer we use is a Kantorovich potential for an optimal transport problem of mapping a distribution of noisy images onto clean ones, as first proposed by Lunz, {\"O}ktem and Sch{\"o}nlieb. We show that the effect of their restoration method on the distribution of the images is an explicit Euler discretization of a gradient flow on probability space, while our variational problem, dubbed Wasserstein ROF (WROF), is the corresponding implicit discretization. We obtain our geometric characterisation of the solution in this setting by first proving a more general convex analysis theorem for variational problems with solutions characterised by projections. We then use optimal transport arguments to obtain our WROF theorem from this general result, as well as a decomposition of a transport map into large scale "features" and small scale "details", where scale refers to the magnitude of the transport distance. Further, we leverage our theory to analyze two algorithms which iterate WROF. We refer to these as iterative regularization and multiscale transport.  For the former we prove convergence to the clean data. For the latter we produce successive approximations to the target distribution that match it up to finer and finer scales. These algorithms are in complete analogy to well-known effective methods based on ROF for iterative denoising, respectively hierarchical image decomposition. We also obtain an analogue of the Tadmor Nezzar Vese energy identity which decomposes the Wasserstein 2 distance between two measures into a sum of non-negative terms that correspond to transport costs at different scales.
\end{abstract}


\section{Introduction}
\label{sec:multiscale_introduction}%
A well-known classical method for image restoration is the total variation approach of Rudin-Osher-Fatemi (ROF) \cite{rudin1992nonlinear}. In this technique, a noisy image $f \in L^2(\RR^2)$ is restored by solving the problem
\begin{equation}
    \min_{u \in L^2(\RR^2)} \frac{1}{2}\norm{u-f}_{L^2(\RR^2)}^2 + \step \norm{u}_{TV}.\label{prob:ROF}
\end{equation}
Here, $\norm{u}_{TV}$ is the total variation norm of $u$, a regularizer known for promoting smoothness while preserving edges. Related to \eqref{prob:ROF} is the more recent variational denoising method of \cite{lunz2018adversarial}. The important novelty of \cite{lunz2018adversarial} is that it uses a learned regularizer instead of the $TV$-norm to impose regularity. The motivation for this is that one may be able to obtain a more effective regularizer -- and experiments show that this is in fact the case --  by learning it from datasets of noisy and clean images rather than using a hand-crafted one. The particular learned regularizer proposed in \cite{lunz2018adversarial} is a Kantorovich potential $u_0$ for the Wasserstein 1 distance $W_1(\mu,\nu)$, where $\mu$ and $\nu$ are probability distributions of noisy and clean data, respectively, on a compact and convex domain $\Omega\subset \RR^d$. That is, $u_0$ solves the problem
\begin{equation*}
    \sup_{u \in \onelip} \int_\Omega u(x) d\mu(x)- \int_\Omega u(y) d\nu(y),
\end{equation*}
where $\onelip$ is the set of functions with Lipschitz constant $1$ on $\Omega$. The solution $u_0$ is thus incentivized to take large values on the noisy data $\mu$ and small values on the real data $\nu$, justifying its role in restoring a noisy image\footnote{Images are taken as vectors in $\RR^d$ here, unlike \eqref{prob:ROF}, where they are elements of $L^2(\RR^2)$.} $x_0 \sim \mu$ by solving
\begin{equation}
    \min_{x \in \Omega} \frac{1}{2}|x-x_0|^2 + \step u_0(x).\label{prob:backward-euler-on-u}
\end{equation}
Experiments in \cite{lunz2018adversarial} show that  denoising performance is improved by using this learned regularizer as opposed to the $TV$-norm.

The ROF model has been intensively studied and has a well developed and beautiful theory (e.g.~\cite{meyer2001oscillating, caselles2007discontinuity, chambolle2004algorithm, chambolle2016geometric}). Let us briefly outline some of the results in \cite{meyer2001oscillating}. The solution 
$u_\step$ to \eqref{prob:ROF} can be described geometrically as the projection of $0$ onto a certain norm ball of radius $\step$ centred at $f$. Moreover, the wavelet coefficients of the residual $f - u_\step$ satisfy an $\ell_\infty$ bound in terms of $\step$, and an approximate solution to \eqref{prob:ROF} can be obtained via soft thresholding of the wavelet coefficients of $f$. Building on these results, \eqref{prob:ROF} can be solved iteratively to obtain iterative denoising (see \cite{athavale2015multiscale} or Section 7.1 of \cite{scherzer2009variational}) and the non-linear hierarchical image decomposition of \cite{tadmor2004multiscale}. The latter can be viewed as non-linear harmonic analysis of the image into components at finer and finer scale, and the analogy is further strengthened by an elegant corresponding energy equality.

We were motivated by these results for ROF to search for a corresponding theory for a learned regularizer problem related to \eqref{prob:backward-euler-on-u}. The first part of this paper establishes analogous theorems to those of \cite{meyer2001oscillating} for a learned regularizer setting. It also includes a decomposition of a certain transport map into large scale ``features'' and small scale ``details''; in this context, scale refers to the magnitude of the transport distance. The second part of the paper leverages our results to analyze two natural iterative optimal transport procedures. We refer to these as iterative regularization and multiscale transport, as they are in correspondence with iterative denoising with ROF and the multiscale image decomposition of \cite{tadmor2004multiscale}. For the former, we prove convergence towards the clean data distribution $\nu$. The latter has a richer structure, and modifies $\nu$ at each stage to obtain a ``sketch'' of $\mu$ which is indistinguishable from it up to a pre-defined scale. Our results in this direction also include an energy identity analogous to that of \cite{tadmor2004multiscale} which decomposes the squared Wasserstein 2 distance $W_2^2(\mu,\nu)$ into a sum of non-negative terms which picks out the scales of transport.

While \eqref{prob:backward-euler-on-u} is a pointwise formulation of image restoration, the setting is more global in that $u_0$ depends on the distribution $\mu$ and $\nu$ of noisy and clean images. We have thus found it more natural to analyse the measure obtained by modifying $\mu$ with the solution map to \eqref{prob:backward-euler-on-u}. Taking  this as a starting point, the main object of study in this paper is 
\begin{equation}
\inf_{\rho \in \mathcal{P}(\Omega)} \frac{1}{2}W_2^2(\mu,\rho) + \step W_1(\rho, \nu).\tag{WROF}\label{prob:initial-multiscale-statement}
\end{equation}
Here $\mathcal{P}(\Omega)$ is the space of Borel probability measures on $\Omega$, and for $p \geq 1$, $W_p: \mathcal{P}(\Omega) \times \mathcal{P}(\Omega) \rightarrow \RR$ is the Wasserstein $p$ distance; for more background on optimal transport we refer the reader to \cite{santambrogio2015optimal} or \cite{villani2009optimal}. Given that $\mu$ consists of noisy images, and $\nu$ is a distribution of clean images, we view $\frac{1}{2}W_2^2(\mu,\rho)$ as a fidelity term while $W_1(\rho, \nu)$ measures regularity. As we will see in \Cref{prop:meyer-for-multiscale-simplified} and \Cref{thm:analogue_of_wavelet_shrinkage}, this problem has properties which are in exact correspondence with the aforementioned results for ROF. As a consequence we call it Wassertein ROF (or WROF for short).

To motivate the study of \eqref{prob:initial-multiscale-statement}, let us specify its relationship to the image denoising technique of \cite{lunz2018adversarial}. We will show, in \Cref{lem:char-of-solutionmap}, that the measure obtained by pushing $\mu$ forward under the solution map of \eqref{prob:backward-euler-on-u} is the unique solution to 
\begin{equation}
    \inf_{\rho \in \mathcal{P}(\Omega)} \frac{1}{2}W_2^2(\rho, \mu) + \step\langle u_0, \rho \rangle.\label{prob:forward_euler_inW2}
\end{equation}
Since $u_0$ is a sub-gradient of the convex functional $\mu \mapsto W_1(\mu,\nu)$, \eqref{prob:forward_euler_inW2} can be viewed as an explicit Euler discretization of a gradient flow on the space $\mathbb{W}_2(\Omega)$ of probability distributions metrized by the Wasserstein 2 distance. A step of the implicit Euler discretization of the same flow is \eqref{prob:initial-multiscale-statement}. We focus on \eqref{prob:initial-multiscale-statement}, as opposed to \eqref{prob:forward_euler_inW2}, because in general the implicit method has better properties than the explicit one. We note, however, that in certain settings the two approaches coincide (see \Cref{lem:multiscale_generalizes_TTC}). In addition, the implicit Euler approach retains a pointwise reconstruction method; there is a continuous function $\varphi_\step$ such that the solution $\rho_\step$ to \eqref{prob:initial-multiscale-statement} is obtained by modifying $\mu$ pointwise by the solution map for
\begin{equation}
    \inf_{x \in \Omega} \frac{1}{2}|x-x_0|^2 - \varphi_\step(x). \label{eq:different_learned_reg}
\end{equation}
In fact, $\varphi_\step$ is a Kantorovich potential for the transport from $\mu$ to $\nu$ under the cost function $\hybrid$ defined in \eqref{eq:hybridcost} (see \Cref{prop:relation_between_Tstep_andvarphistep}). In this sense, the solution to \eqref{prob:initial-multiscale-statement} is obtained via restoration with a learned regularizer $-\varphi_\step$. Moreover, $\varphi_\step$ can be taken so that $\frac{1}{2}|x|^2 - \varphi_\step(x)$ is convex, which implies that the pointwise restoration algorithm \eqref{eq:different_learned_reg} has the additional benefit of being a convex optimization problem; in this light, \eqref{eq:different_learned_reg} bears a similarity to the convex learned regularizers of \cite{mukherjee2020learned}. We also suspect that restoration via \eqref{eq:different_learned_reg} may be more effective than \eqref{prob:backward-euler-on-u}, since \Cref{prop:convergence_to_nu} shows that iterations of this procedure provably converge to the clean image distribution $\nu$.

In the remainder of this section we will summarize our main results, with \Cref{sec:intro-characterisation-of-solution} describing our geometric characterisation of the solution of \eqref{prob:initial-multiscale-statement}, while \Cref{sec:intro-iterative_regularization} and \Cref{sec:intro-nonlinear_Plancherel} outline our iterative procedures.

\subsection{Geometric characterisation of the solution of \eqref{prob:initial-multiscale-statement}}
\label{sec:intro-characterisation-of-solution}
In this section we provide analogues in the setting of a learned regularizer of results giving a geometric characterisation of the solution to ROF.

First, we recall some classical results for ROF. In studying this problem, it is helpful to define the dual norm to $\norm{\cdot}_{TV}$; for $v \in L^2(\RR^2)$, define the $*$-norm as
\begin{equation}
    \norm{v}_* = \sup \{ \int_{\RR^2} v u dx \mid \norm{u}_{TV} \leq 1\}.\label{eq:star_norm}
\end{equation}
The following theorem, mentioned in \Cref{sec:multiscale_introduction}, is a slight reformulation of results from \cite{meyer2001oscillating} on the solution to \eqref{prob:ROF}. Specifically, it characterises the solution as a projection of $0$ onto a ball in the $*$-norm centred at $f$. 
\begin{theorem}[Meyer]
\label{cor:meyer_for_ROF}
For all $\step >0$, \eqref{prob:ROF} has a unique solution $u_\step$, which can also be expressed as the solution to
\begin{equation}
    \min_{ \norm{u-f}_* \leq \step} \norm{u}^2_{L^2(\RR^2)}.\label{prob:ROF-projection}
\end{equation}
Consequently, if $\norm{f}_* \leq \step$,  $u_\step= 0$. On the other hand, if $\norm{f}_* > \step$, then $\norm{f-u_\step}_* = \step$ and
\begin{equation*}
    \int_{\RR^2} u_\step (f-u_\step)dx = \step \norm{u}_{TV}.
\end{equation*}
\end{theorem}
\begin{remark}
\label{remark:wavelet-shrinkage}
\Cref{cor:meyer_for_ROF} provides a formal statement of some of the results we have mentioned in \Cref{sec:multiscale_introduction}. For a statement of further results on ROF, such as the $\ell_\infty$ bound on the wavelet coefficients of $f-u_\step$ or the fact that an approximate solution can be obtained by applying soft-thresholding to the wavelet coefficients of $f$, see \cite{meyer2001oscillating}, Lemma 10, Section 1.14.
\end{remark}

Our \Cref{prop:meyer-for-multiscale-simplified} gives analogous results for \eqref{prob:initial-multiscale-statement}. To make the analogy clear, \Cref{tab:analogy} gives the correspondence between the key concepts. In this case, the measure $\nu$ is projected with respect to a divergence $D_\step$ onto a set of measures $B_\step(\mu)$. We will be more precise about $D_\step$ and $B_\step(\mu)$ in \eqref{def:D_deff} and \eqref{eq:ball_def}. We will see that these notions are natural from the point of view of convex analysis; for now, we describe them in intuitive terms.

A key role will be played by an optimal transport problem that uses a cost function $\hybrid: \Omega \times \Omega \rightarrow \RR$ related to the Huber loss function \cite{huber1964robust} for robust estimation. It is given by
\begin{equation}
\label{eq:hybridcost}
\hybrid(x,y) = \begin{cases}
\frac{1}{2}|x-y|^2 &\quad |x-y| \leq \step,\\
\step|x-y| -\frac{\step^2}{2} &\quad |x-y| \geq \step.
\end{cases}
\end{equation}
This can be viewed as a variation on the standard cost function $c_2(x,y) = \frac{1}{2}|x-y|^2$, except with a certain economy of scale; in particular, the cost of transport at distances larger than $\step$ is discounted. This may be advantageous for image restoration since this cost is robust to outliers. The relationship between the solution $\rho_\step$ to \eqref{prob:initial-multiscale-statement} and an optimal plan transporting $\mu$ to $\nu$ under the cost $\hybrid$ will be made explicit in \Cref{prop:relation_between_Tstep_andvarphistep}. We also note that the minimum value of \eqref{prob:initial-multiscale-statement} is the optimal transport cost from $\mu$ to $\nu$ for the pointwise cost $\hybrid$; see \Cref{cor:wrof-is-hybrid}.

The set $B_\step(\mu)$ consists of measures which can be reached from $\mu$ with displacement less than $\step$ by an optimal transport plan for the cost $\hybrid$. In this sense, measures in $B_\step(\mu)$ are indistinguishable from $\mu$ up to scale $\step$. 

The divergence $D_\step(\nu,\rho)$ is non-negative, and is $0$ only when $\rho= \nu$ provided $\mu$ is absolutely continuous with respect to Lebesgue measure, which we denote by $\mu \ll \lebesgue$. Further, we will show that $D_\step(\nu, \rho)$ has an interesting economic interpretation. In short, assuming that goods are sold to consumers with distribution $\nu$ and purchased from a manufacturer with distribution $\rho$, $D_\step(\nu,\rho)$ represents the total loss of value in a supply chain when the transport cost has an economy of scale and consumers adopt a ``buy local'' policy. More concretely, at the optimal $\rho_\step$ for \eqref{prob:initial-multiscale-statement}, $D_\step(\nu, \rho_\step)$ measures the amount of transport between $\mu$ and $\nu$ at scale larger than $\lambda$; our results (specifically \Cref{prop:meyer-for-multiscale}, together with \Cref{cor:wrof-is-hybrid}), imply
    \begin{equation}
        \int_{\Omega^2}\frac{1}{2}(|x-y|-\step)^2_+ d\tilde{\gamma}_{0} \geq D_{\step}(\nu, \rho_\step)
        \geq \int_{\Omega^2}\frac{1}{2}(|x-y|-\step)^2_+ d\gamma_{0},\label{eq:D_interp}
    \end{equation}
    where $\tilde{\gamma}_{0}$ and $\gamma_0$ are optimal plans for transporting $\mu$ to $\nu$ under the costs $\hybrid$ and $c_2$, respectively. 
    
Analogously to \Cref{cor:meyer_for_ROF}, our first theorem expresses the solution to \eqref{prob:initial-multiscale-statement} as a projection of $\nu$ onto $B_\step(\mu)$. We also include an additional result (see \eqref{eq:displacement_bound_for_w2optimal}) which is analogous to the $\ell_\infty$ bound on the wavelet coefficients of the residual $f - u_\step$ mentioned in \Cref{remark:wavelet-shrinkage}.
\begin{theorem}[Main theorem, part 1]
\label{prop:meyer-for-multiscale-simplified}
Let $\Omega$ be compact and convex with non-negligible interior, and suppose $\mu \ll \mathcal{L}_d$. For all $\step>0$, \eqref{prob:initial-multiscale-statement} has a unique solution $\rho_\step$, which can also be expressed as the solution to
\begin{equation}
     \min_{\rho \in B_\step(\mu)}\bregmandiv{}(\nu, \rho)\label{prob:WROF_projection}
\end{equation}
Consequently, if $\nu \in B_\step(\mu)$, $\rho_\step = \nu$. On the other hand, if $\nu \not \in B_\step(\mu)$, then there exists $\varphi_\step$ a Kantorovich potential for $W_2(\mu,\rho_\step)$ satisfying $\text{Lip}(\varphi_\step) = \step$ and
\begin{equation}
    \int_\Omega \varphi_\step(d\nu - d\rho_\step) = \step W_1(\rho_\step, \nu).
\end{equation}
Finally, the optimal transport map $T_\lambda$ for $W_2(\mu,\rho_\step)$ satisfies
    \begin{equation}
        \norm{I-T_\lambda}_{L^\infty(\mu)} \leq \step.\label{eq:displacement_bound_for_w2optimal}
    \end{equation}   
\end{theorem}
\begin{table}
\caption{The analogy between \eqref{prob:ROF} and \eqref{prob:initial-multiscale-statement}. The decompositions of $f$ and $S_0$ are described in \eqref{eq:ROF-decomp} and \eqref{eq:decomposition_of_transport}, respectively.\label{tab:analogy}}
\begin{center}
\begin{tabular}{c|c|c}
   & ROF & WROF  \\
   Fidelity  & $\norm{u-f}^2_{L^2(\RR^2)}$ & $W_2^2(\rho, \mu)$ \\
   Regularity & $\norm{u}_{TV}$ & $W_1(\rho,\nu)$ \\
   Projection Metric  & $\norm{u}^2_{L^2(\RR^2)}$ &$D_\step(\nu, \rho)$\\
   Projection Set & $\{ u \mid \norm{u-f}_* \leq \step\}$ & $B_\step(\mu)$ \\
   Decomposition & $f = v_\step + u_\step$ & $S_0 = T_\step^{-1} \circ S_\step$ 
\end{tabular}
\end{center}
\end{table}
A more detailed version of this result is given in \Cref{prop:meyer-for-multiscale}. In \Cref{sec:deeper_analysis} and \Cref{sec:applying_general_to_WROF} we will clarify the strong similarities between \Cref{cor:meyer_for_ROF} and \Cref{prop:meyer-for-multiscale-simplified}  by proving a general theorem for a class of convex optimization problems of the form \eqref{prob:general_minimization_problem} for which the solution map is a projection. We will show that ROF and \eqref{prob:initial-multiscale-statement} are included in this class, so that \Cref{cor:meyer_for_ROF} and \Cref{prop:meyer-for-multiscale-simplified} will follow as particular cases. 

More insight into $\varphi_\step$ and $T_\step$ from \Cref{prop:meyer-for-multiscale-simplified} is given in the following proposition.
\begin{proposition}
\label{prop:relation_between_Tstep_andvarphistep}
    Under the notation and assumptions of \Cref{prop:meyer-for-multiscale-simplified}, 
    \begin{enumerate}
        \item $\varphi_\step$ is a solution to
        \begin{equation*}
            \sup_{\varphi \in C(\Omega)} \int_\Omega \varphi^{\hybrid} d\mu + \int_\Omega \varphi d\nu,
        \end{equation*}
        where $\varphi^{\hybrid}(x) = \inf_{y \in \Omega} \hybrid(x,y) - \phi(y)$,
        \item $T_\step$, which by definition satisfies $(T_\step)_\# \mu = \rho_\step$, is the solution map to \eqref{eq:different_learned_reg}, and
        \item if $\gamma_0$ is an optimal transport plan for transporting $\mu$ to $\nu$ under the cost $\hybrid$, and if $(x,y) \in \spt(\gamma_0)$, then
        \begin{equation}
            T_\step(x) = \begin{cases}
                y &\quad |x-y|\leq \step,\\
                \left(1- \frac{\lambda}{|x-y|}\right) x + \frac{\lambda}{|x-y|} y &\quad |x-y|> \step.
            \end{cases}\label{eq:restoration_map}
        \end{equation}
    \end{enumerate}
\end{proposition}
\begin{remark}
\label{rem:pointwise_restoration}
\Cref{prop:relation_between_Tstep_andvarphistep} shows precisely the outcome $T_\lambda (x_0)$ of restoring a noisy image $x_0$ by solving \eqref{eq:different_learned_reg} with the learned regularizer $\varphi_\step$. The answer is determined by $\gamma_0$; if $(x_0,y_0) \in \spt(\gamma_0)$ is such that $|x_0-y_0|\leq \step$, $T_\step$ completes the transport from $x_0$ to $y_0$. On the other hand, if $|x_0-y_0| > \step$, $T_\step$ takes a step of size $\step$ in the direction of $y_0$. 
\end{remark}

Assuming that $\nu$ is also absolutely continuous, we further establish in the following theorem that $\rho_\step$ is obtained by applying soft thresholding to an optimal transport map from $\nu$ to $\mu$. Recall that the soft thresholding map is given by $s_\step: \RR \rightarrow \RR$,
\begin{equation*}
    s_{\step}(t) := \text{sign}(t)(|t|-\step)_+.
\end{equation*}
This provides an analogous result to the soft thresholding property of ROF mentioned in \Cref{remark:wavelet-shrinkage}, except that here we obtain the exact solution rather than an approximate one.

\begin{theorem}[Main theorem, part 2]
\label{thm:analogue_of_wavelet_shrinkage}
In addition to the hypotheses of \Cref{prop:meyer-for-multiscale-simplified}, assume that $\nu \ll \lebesgue$. Then
\begin{enumerate}
     \item $\rho_\step \ll \lebesgue$,\label{claim:rho_step_AC}
      \item $S_0$ is an optimal transport map for the cost $\hybrid$ sending $\nu$ to $\mu$ if and only if
    \begin{equation}
        S_0 = T_\step^{-1} \circ S_\step\label{eq:decomposition_of_transport}
    \end{equation}
    where $T_\step^{-1}$ is a Borel map satisfying $T_\step^{-1}\circ T_\step (x) = x$ $\mu$ almost everywhere, and $S_\step$ is an optimal transport map for $W_1(\nu,\rho_\step)$.
    \label{claim:existence_of_huber_map}
    \item For any such $S_0$, the solution $\rho_\step$ to \eqref{prob:initial-multiscale-statement} is obtained as $\rho_\step = (S_\step)_\# \nu$, where \label{claim:rho_step_obt_from_ST}
    \begin{equation}
        S_\step(y):= y + s_{\step}(|S_0(y)-y|)\frac{S_0(y)-y}{|S_0(y) -y|}.\label{eq:soft_thresholding_tmap}
    \end{equation}
\end{enumerate}
\end{theorem}
\begin{remark}
This result gives a further interpretation of $\step$ as a scale parameter, in the sense that the solution $\rho_\lambda$ to \eqref{prob:initial-multiscale-statement} is obtained from $\nu$ by only transporting mass that moves larger than distance $\step$ under $S_0$. The formula \eqref{eq:decomposition_of_transport} also deepens the analogy to ROF. Recall that, writing the residual $f - u_\step$ as $v_\step$, ROF provides a decomposition of the image $f$ into ``features'' $u_\step$ and ``details'' $v_\step$, connected by the formula
\begin{equation}
    f = v_\step + u_\step.\label{eq:ROF-decomp}
\end{equation}
The equation \eqref{eq:decomposition_of_transport} is an optimal transport analogue of this decomposition, the analogy being obtained by replacing addition with composition. Thus, the transport map $S_0$ is decomposed into $S_\lambda$ (which we think of as features in the sense that it only involves large scale transport) and details $T_\lambda^{-1}$ which only involve transport less than distance $\lambda$ (see \eqref{eq:displacement_bound_for_w2optimal}). This decomposition will be analysed in detail in \Cref{sec:existence_of_huber_map}.
\end{remark}


\subsection{Iterative regularization}
\label{sec:intro-iterative_regularization}
We now move to a description of the results in the second part of the paper, and introduce our first iterative procedure. It is in correspondence with iterated denoising through repeated applications of ROF (see \cite{athavale2015multiscale} or Section 7.1 of \cite{scherzer2009variational}). Here we study iterations of the problem \eqref{prob:initial-multiscale-statement}, where at each stage $\mu$ is replaced with the previous solution $\rho_\step$. When $\mu$ is a distribution of noisy images and $\nu$ is a distribution of clean ones, this represents the iterative regularization of $\mu$. The following proposition is our main result in this direction.

\begin{proposition}
\label{prop:convergence_to_nu}
Let $\Omega$ be convex and compact with non-negligible interior. Let $\mu,\nu \ll \mathcal{L}_d$, and suppose that $(\step_n)_{n=0}^\infty$ is a sequence of positive step sizes with
\begin{equation}
    \sum_{n=0}^\infty \step_n = +\infty.\label{eq:unbounded_eta_sum}
\end{equation}
Given $\mu_0 := \mu$, for each $n\geq 0$ define
\begin{equation}
    \mu_{n+1} := \argmin_{\rho \in \mathcal{P}(\Omega)} \frac{1}{2}W_2^2(\rho, \mu_{n}) + \step_{n} W_1(\rho, \nu). \label{def:mu_n_def}
\end{equation}
Then
\begin{equation}
    \lim_{n\rightarrow \infty} W_1(\mu_n, \nu) = 0.\label{eq:convergence_to_nu}
\end{equation}
\end{proposition}
We note that due to statement \ref{claim:rho_step_AC} of \Cref{thm:analogue_of_wavelet_shrinkage}, if $\mu$ and $\nu$ are absolutely continuous then the solution to \eqref{prob:initial-multiscale-statement} is absolutely continuous as well. In connection with \Cref{prop:meyer-for-multiscale-simplified}, this guarantees that the argmin in \eqref{def:mu_n_def} is unique for each $n$, establishing that the sequence $\mu_n$ is well defined.

\subsection{Multiscale transport and a non-linear energy decomposition}
\label{sec:intro-nonlinear_Plancherel}

Our second iterative process proceeds in the other direction (i.e. "adding detail" as opposed to denoising), and reveals a richer structure. It is analogous to the hierarchical image decomposition from \cite{tadmor2004multiscale}, and so we first briefly recall those results here. This approach leverages ROF to decompose an image $f$ into a hierarchical representation $(u_n)_{n=1}^\infty$ of features at different scales by setting
\begin{equation}
    u_{n+1} := \argmin_{u \in L^2(\RR^2)} \norm{u-v_n}_{L^2(\RR^2)}^2 + \step_{n+1} \norm{u}_{TV}, \quad v_n = f - \sum_{i=1}^n u_i, \label{prob:multiscale}
\end{equation}
where $v_0:= f$ and $\step_n = 2^{-n+1} \step_1$. Thus, at each stage the ``detail'' component $v_n$ is broken down into smaller scale features $u_{n+1}$ and details $v_{n+1}$. The following theorem\footnote{\cite{tadmor2004multiscale} included this result for the cases $f \in BV(\RR^2)$ or $f$ in an intermediate space between $BV(\RR^2)$ and $L^2(\RR^2)$. A proof requiring only $f \in L^2(\RR^2)$ was obtained in \cite{modin2019multiscale}} establishes that $(u_n)_{n=1}^\infty$ is indeed a decomposition of $f$ and provides a non-linear harmonic analysis identity for $\norm{f}_{L^2(\RR^2)}^2$. 

\begin{theorem}[from \cite{modin2019multiscale}]
\label{thm:multiscale_for_ROF}For $f \in L^2(\RR^2)$, the sequence $(u_n)_{n=1}^\infty$ defined by \eqref{prob:multiscale} satisfies
\begin{equation}
    f = \sum_{n=1}^\infty u_n,\label{eq:decomposition_of_f}
\end{equation}
where the convergence holds in the strong sense in $L^2(\RR^2)$. Further,
\begin{equation}
    \norm{f}_{L^2(\RR^2)}^2 = \sum_{n=1}^\infty  \norm{u_n}^2_{L^2(\RR^2)} + \step_{n} \norm{u_n}_{TV}.\label{eq:nonlin_plancherel}
\end{equation}
\end{theorem}
More insight on the scale of the decomposition $(u_n)_{n=1}^\infty$ can be obtained from \Cref{cor:meyer_for_ROF}, which states that
\begin{equation}
    \norm{f - \sum_{i=1}^n u_i}_* \leq \frac{\step_1}{2^{n}}.\label{eq:scale_estimate_for_f}
\end{equation}
Thus $f$ and the partial sum $\sum_{i=1}^n u_i$ agree up to a term of scale at most $2^{-n}\step_1$ in the norm $\norm{\cdot}_*$. As we have mentioned, according to \cite{meyer2001oscillating}, Lemma 10, Section 1.14 this puts an $\ell^\infty$ bound on the wavelet coefficients of $f - \sum_{i=1}^n u_i$. 

By analogy to this approach, our iterative process evolves by leaving $\mu$ untouched at each step and replacing $\nu$ with the previous iterate, $\nu_n$; the manner in which this is analogous to \eqref{prob:multiscale} will be made precise in \Cref{remark:how_is_it_analogous?}. We describe this procedure as ``adding detail'' since by solving \eqref{prob:initial-multiscale-statement} with a large value of $\step$ we obtain a modification of $\nu$ which is a ``sketch'' of $\mu$, in that the two measures are indistinguishable up to transport at scale $\step$ (see \Cref{prop:meyer-for-multiscale-simplified}). By repeating this process with a smaller value of $\step$ we refine this sketch, obtaining at each stage finer details of $\mu$. Note also that under the additional assumption $\nu \ll \lebesgue$, \Cref{thm:analogue_of_wavelet_shrinkage} implies that we are decomposing a transport map at each stage of this procedure into ``features'' and ''details'', as determined by the scale of the transport relative to $\step$. Due to the soft thresholding (see \eqref{eq:soft_thresholding_tmap}), the latter are untouched, to be resolved at future steps, while the former are partially carried out until the remaining transport becomes a detail. 

Finally, we obtain in \eqref{eq:plancherel_for_measures} a decomposition of the total energy $W_2^2(\mu,\nu)$ which includes all the scales of transport from $\nu$ to $\mu$ via \eqref{eq:D_interp}; this is in correspondence with the identity \eqref{eq:nonlin_plancherel}.

The following proposition summarizes the properties of this multiscale algorithm which are not directly implied by \Cref{prop:meyer-for-multiscale-simplified} or \Cref{thm:analogue_of_wavelet_shrinkage}. Note that we do not require $\nu \ll \lebesgue$ for these results.
\begin{theorem}
\label{thm:multiscale}
Let $\Omega \subset \RR^d$ be compact and convex with a non-negligible interior. Take $\mu,\nu \in \mathcal{P}(\Omega)$ with $\mu \ll \lebesgue$. Suppose $\step_0$ is given. For each $n \geq 0$, set $\step_{n+1} = \step_n/2$ and define
\begin{equation}
    \nu_{n+1} := \argmin_{\rho \in \mathcal{P}(\Omega)}\squaredw{\rho}{\mu} + \step_{n} W_1(\rho, \nu_{n}),\label{eq:multiscale_for_measures}
\end{equation}
where $\nu_0 := \nu$. We have that
\begin{enumerate}
    \item\label{claim:new_convergence} The sequence $\nu_n$ converges to $\mu$ with rate
    \begin{equation}
        \squaredw{\mu}{\nu_n} \leq 2^{-2n+1}\step_0^2,\label{eq:convergence_of_nun}
    \end{equation}
    and,
    \item \label{claim:plancherel} The following energy equality holds
    \begin{equation}
        \squaredw{\nu}{\mu} = \sum_{n=0}^\infty D_{\step_{n}}(\nu_{n}, \nu_{n+1}) + \step_n W_1(\nu_n,\nu_{n+1}).\label{eq:plancherel_for_measures}
    \end{equation}
\end{enumerate}
\end{theorem}
\begin{remark}
If we add the assumption that $\nu$ is absolutely continuous, we obtain that the measures $\nu_n$ specified in \Cref{thm:multiscale} can be written as $(S_{\step_{n-1}} \circ \cdots \circ S_{\step_0})_\# \nu$; see \Cref{thm:analogue_of_wavelet_shrinkage}. In this way, $\nu_n$ is built up from a composition of Wasserstein 1 optimal maps  applied to $\nu$. In this sense we are replacing the summation of the decomposition in \eqref{eq:decomposition_of_f} with composition, as was done for a multiscale decomposition of diffeomorphisms in \cite{modin2019multiscale}.
\end{remark}
We now describe the organization of the paper. We provide in \Cref{sec:multiscale_discussion} a discussion of related work. Then, in \Cref{sec:motivation-for-multiscale} we elucidate the connection between \eqref{prob:initial-multiscale-statement} and the restoration via learned regularizer technique of \cite{lunz2018adversarial}. In \Cref{sec:deeper_analysis} we introduce a general class of optimization problems (which includes both ROF and \eqref{prob:initial-multiscale-statement}), 
 and prove \Cref{thm:general_meyer} characterising their solution maps as projections.  In \Cref{sec:applying_general_to_WROF}, we use optimal transport arguments to obtain \Cref{prop:meyer-for-multiscale-simplified} from
\Cref{thm:general_meyer}. In \Cref{sec:absolute_continuity_of_rho0} we prove that the solution to \eqref{prob:initial-multiscale-statement} is absolutely continuous if $\mu$ and $\nu$ are, and in \Cref{sec:existence_of_huber_map} we prove the existence of an optimal transport map from $\nu$ to $\mu$ under the Huber cost $\hybrid$, as well as the soft thresholding formula \eqref{eq:soft_thresholding_tmap}. Together, \Cref{sec:absolute_continuity_of_rho0} and \Cref{sec:existence_of_huber_map} prove \Cref{thm:analogue_of_wavelet_shrinkage}. Finally, the results for our iterative procedures (i.e. \Cref{prop:convergence_to_nu} and \Cref{thm:multiscale}) are proved in \Cref{sec:iterative_procs}.

\section{Discussion and related work}
\label{sec:multiscale_discussion}
There is a connection between our iterative regularization procedure defined in \Cref{prop:convergence_to_nu} and the JKO scheme \cite{jordan1998variational}. The latter is related to gradient flows in $\mathbb{W}_2(\Omega)$, which are analysed in more detail in \cite{ambrosio2005gradient} (see also \cite{santambrogio2017euclidean}). The JKO algorithm produces a sequence of measures $\rho_n$ by iteratively solving an equation of the type
\begin{equation}
    \rho_n:=\argmin_{\rho \in \mathcal{P}(\Omega)}\frac{1}{2\step}W_2^2(\rho_{n-1}, \rho) + F(\rho),\label{eq:general_gradient_flow}
\end{equation}
where $F$ is a functional. For $F(\rho) = W_1(\rho,\nu)$, this problem is precisely \eqref{prob:initial-multiscale-statement}. For general $F$, by allowing $\step$ to go to zero and examining the optimality conditions of \eqref{eq:general_gradient_flow}, one can obtain convergence of an interpolation of the iterates $\rho_n$ to a curve of measures $\rho(t)$. This curve satisfies a PDE which can be viewed as a gradient flow on $F$ in the metric space $\mathbb{W}_2(\Omega)$.
We expect the PDE that corresponds to our iterative denoising algorithm to be of the form
\begin{equation}
    \partial_t \rho(t) - \nabla \cdot (\rho(t)\nabla u_0(t)) = 0,
\end{equation}
where for all $t$, $u_0(t)$ is a Kantorovich potential for $W_1(\rho(t), \nu)$. We leave the rigorous derivation to a separate paper. Note that by analogy to ROF such a flow would be in correspondence with the TV flow in \cite{athavale2015multiscale}.

Other problems of a form similar to \eqref{prob:initial-multiscale-statement} have been considered in the literature. A notable example is \cite{burger2012regularized}, which finds a smoothed version of a probability measure $\mu$ while retaining edges by solving
\begin{equation*}
    \min_{\rho \in \mathcal{P}(\Omega)}\frac{1}{2\step}W_2^2(\mu, \rho) + F(\rho), \quad F(\rho):= \begin{cases}
    \norm{\rho}_{TV} &\quad \rho = \rho(x) dx,\\
    +\infty &\quad \text{ else}.
    \end{cases}
\end{equation*}
A related problem is that of \cite{lellmann2014imaging}, which keeps $F$ as the total variation norm of a probability density but replaces the fidelity term $\frac{1}{2\step}W_2^2(\mu, \rho)$ with the Kantorovich-Rubinstein norm, a quantity that is closely related to the Wasserstein 1 distance, but is able to handle measures with different mass. To our knowledge the specific problem given in \eqref{prob:initial-multiscale-statement} has not been treated before in the literature. Given that previous works have used  the $TV$ norm of a probability density function as a regularity term, we briefly compare this to our approach of using $W_1(\rho,\nu)$ in the particular case of $\nu$ as the normalized Lebesgue measure. One might imagine that for this choice of $\nu$, $W_1(\rho,\nu)$ would serve a similar role to the TV norm, since it is the minimal amount of work required to ``smooth out'' $\rho$ to the constant function. This is not the case, however. Take $\Omega = [0,1]^2$, with $\rho = \rho_k(x) dx$ given by
\begin{equation*}
    \rho_k(x_1,x_2) = 2(1 + \text{sign}(\sin(2 \pi k x_1))).
\end{equation*}
As $k\rightarrow \infty$, $W_1(\rho_k,\nu) \rightarrow 0$, and yet $\norm{\rho_k}_{TV} \rightarrow + \infty$. So the two regularizers play different roles.

The field of image restoration with learned regularizers is rapidly developing, and there are many interesting approaches (e.g. \cite{heaton2022wasserstein, kobler2020total, kobler2017variational, lunz2018adversarial, mukherjee2021end}). We focus on \cite{lunz2018adversarial} as we found it to be a natural and compelling analogue of ROF. Note that \cite{lunz2018adversarial} includes several theoretical results, which focus on issues such as stability of the reconstruction method and a geometric formula for the Kantorovich potential $u_0$ under certain conditions. Let us also note that \cite{mukherjee2021end}, being related to iterations of the method from \cite{lunz2018adversarial}, forms a parallel approach to our iterated regularization discussed in \Cref{sec:intro-iterative_regularization}.

Lastly, numerical results for either of the procedures outlined in \Cref{sec:intro-iterative_regularization} or \Cref{sec:intro-nonlinear_Plancherel} could be obtained using the dual problem (see \Cref{sec:prelim-for-proving-meyer}), 
\begin{equation}
    \sup_{\varphi \in \steplip} \int_\Omega \varphi^{c_2} d\mu + \int_\Omega \varphi d\nu,\label{prob:related_work_dual}
\end{equation}
where $\steplip$ is the set of Lipschitz continuous functions on $\Omega$ with constant $\step$, and $\varphi^{c_2}$ is the $c_2$ transform of $\varphi$, defined in \Cref{remark:c-concave} below. Indeed, \Cref{prop:meyer-for-multiscale} shows that the solution $\rho_\step$ to \eqref{prob:initial-multiscale-statement} can be realized by applying the solution map to \eqref{eq:different_learned_reg} pointwise to $\mu$, where $\varphi_\step$ solves \eqref{prob:related_work_dual}. By analogy to \cite{gulrajani2017improved}, it is natural to obtain such a $\varphi_\step$ by parametrizing it with a neural network $\varphi_w$ with weights $w$ and solving the gradient penalty problem
\begin{equation*}
    \sup_w \int_\Omega \varphi_w^{c_2}(x) d\mu(x) + \int_\Omega \varphi_w(y) d\nu(y) - \frac{\lambda}{2}\int_\Omega (|\nabla \varphi_w|-\step)_+^2 d\sigma(x),
\end{equation*}
for large $\lambda$, where $\sigma$ is the sampling distribution from \cite{gulrajani2017improved}. Optimizing the weights $w$ requires the computation of the $c_2$-transform of $\varphi_w$. A general and efficient numerical
algorithm to do so has been introduced in \cite{jacobs2020fast}, a method specific to neural networks has been given in \cite{makkuva2019optimal}, and a new approach which scales well to high dimensions has recently been proposed in \cite{amos2022amortizing}.

\section{Links between \eqref{prob:initial-multiscale-statement} and denoising by adversarial regularization}
\label{sec:motivation-for-multiscale}
In this section we will study the relationship between \eqref{prob:initial-multiscale-statement} and the denoising technique of \cite{lunz2018adversarial}. We will show in \Cref{subsec:forward-and-backward-euler} that the approach of \cite{lunz2018adversarial} can be viewed as an explicit Euler discretization of the gradient flow on $W_1(\cdot, \nu)$ in the metric space $\mathbb{W}_2(\Omega)$. In contrast, \eqref{prob:initial-multiscale-statement} can be viewed as an implicit Euler discretization of the same flow on the same metric space. Moreover, we will establish in \Cref{subsec:optimality-conditions} that these techniques produce identical measures under the assumption that the minimal displacement of the ray monotone optimal transport map for $W_1(\mu,\nu)$ (see \cite{ambrosio2003existence} or Section 3.1 of \cite{santambrogio2015optimal}) is larger than $\step$.

\subsection{Explicit and Implicit Euler on $\mathbb{W}_2(\Omega)$}
\label{subsec:forward-and-backward-euler}
We begin with \Cref{lem:uniquesolvability}, which states that \eqref{prob:backward-euler-on-u} has a unique solution for almost all $x_0$. This is a standard result; we include the proof for completeness. We first recall the following definition.
\begin{definition}
\label{remark:c-concave}
For a symmetric cost function $c: \Omega \times \Omega \rightarrow \RR$, and $\phi \in C(\Omega)$, the function
\begin{equation*}
    \phi^c(x) = \inf_{y \in \Omega} c(x,y) - \phi(y)
\end{equation*}
is called the $c$-transform of $\phi$. If $\phi$ is such that there exists a function $\psi$ with $\phi = \psi^c$, then one says that $\phi$ is $c$-concave, written $\phi \in c\text{-conc}(\Omega)$. 
\end{definition}
Throughout this paper we will make use of the well known fact that $\phi \leq \phi^{cc}$, with equality if and only if $\phi$ is $c$-concave (see, e.g., \cite{santambrogio2015optimal} Proposition 1.34).
\begin{lemma}
\label{lem:uniquesolvability}
Let $\Omega$ be compact with boundary of Lebesgue measure zero. Let $u_0: \Omega \rightarrow \RR$ be lower semi-continuous. Then for almost all $x \in \Omega$, the problem
\begin{equation}
\min_{y \in \Omega} \frac{1}{2}|x-y|^2 + \step u_0(y)\label{prob:CROF}
\end{equation}
has a unique solution given by $x - \nabla (-\lambda u_0)^{c_2}(x)$.
\end{lemma}
\begin{proof}
Since $\Omega$ is compact and $u_0$ is lower semi-continuous, \eqref{prob:CROF} has a solution for all $x \in \Omega$ and the value of the minimum is finite. Compactness of $\Omega$ also implies that $(-\step u_0)^{c_2}$
is Lipschitz (see, for example, Box 1.8 of \cite{santambrogio2015optimal}), and thus the set of $x_0 \in \Omega\setminus \partial \Omega$ such that $\nabla(-\step u_0)^{c_2}(x_0)$ exists has full Lebesgue measure. 

For $x_0$ selected in this way, let $y_0 \in \Omega$ solve \eqref{prob:CROF}. By definition, for all $x \in \Omega$,
\begin{equation}
(-\step u_0)^{c_2}(x) \leq \frac{1}{2}|x-y_0|^2 + \step u_0(y_0),
\end{equation}
with equality at $x = x_0$. Thus, we obtain that the function $x \mapsto \frac{1}{2}|x-y_0|^2 - (-\step u_0)^{c_2}(x)$ is minimized at $x_0$. We therefore have
\begin{equation}
y_0 = x_0 - \nabla (-\step u_0)^{c_2}(x_0).
\end{equation}
This expresses the minimizer $y_0$ of \eqref{prob:CROF} for $x = x_0$ explicitly in terms of $x_0$; the minimizer is therefore unique.
\end{proof}

\Cref{lem:uniquesolvability} implies that whenever $\mu \ll \lebesgue$ and $u_0$ is continuous, \eqref{prob:CROF} has a unique solution $\mu$ almost everywhere, given by $(I-\nabla (-\step u_0)^{c_2})(x_0)$. The following lemma characterises the measure we obtain if we push $\mu$ forward under this solution map.


\begin{lemma}
\label{lem:char-of-solutionmap}
In addition to the assumptions of \Cref{lem:uniquesolvability}, let $\mu \in \mathcal{P}(\Omega)$ satisfy $\mu \ll \lebesgue$. Let $T$ be a Borel map which coincides with $I-\nabla (-\step u_0)^{c_2}$, $\mu$ almost everywhere. Then the measure $T_\# \mu$ is the unique solution to the optimization problem
\begin{equation}
    \inf_{\rho \in \mathcal{P}(\Omega)} \frac{1}{2}W_2^2(\rho, \mu) + \step\langle u_0, \rho \rangle.\label{prob:forwardeuler}
\end{equation} 
\end{lemma}
\begin{proof}
First we note that the map 
\begin{equation}
\rho \mapsto \frac{1}{2}W_2^2(\rho, \mu) + \step\langle u_0, \rho \rangle \label{eq:W2tomu}
\end{equation}
is strictly convex by Theorem 7.19 from \cite{santambrogio2015optimal}, which holds since $\mu$ is absolutely continuous and $\lebesgue(\partial \Omega) = 0$. Thus, if a solution $\rho_0$ to \eqref{prob:forwardeuler} exists it is unique. A measure $\rho_0$ is a minimizer of \eqref{prob:forwardeuler} if and only if
\begin{equation*}
    0 \in \partial \left(\frac{1}{2}W_2^2(\cdot, \mu) + \langle \step u_0, \cdot \rangle\right)(\rho_0)
\end{equation*}
Since $\rho \mapsto \langle u_0, \rho\rangle$ is linear, this is equivalent to
\begin{equation*}
    -\step u_0 \in \partial \left(\frac{1}{2}W_2^2(\cdot, \mu)\right)(\rho_0).
\end{equation*}
By Proposition 7.17 of \cite{santambrogio2015optimal}, which characterises the subdifferential of the convex function $\rho \mapsto \frac{1}{2}W_2^2(\rho, \mu)$, we conclude that $\rho_0$ is a minimizer of \eqref{prob:forwardeuler} if and only if
\begin{equation}
    \int_\Omega (-\step u_0)^{c_2}d\mu + \int_\Omega (-\step u_0) d\rho_0 = \frac{1}{2}W_2^2(\mu, \rho_0). \label{eq:sufficient-cond-for-min-forwardeuler}
\end{equation}
This equality will be proved in in \Cref{lem:subdiff_of_H} for $\rho_0 = T_\# \mu$, when $T$ is $\mu$ almost everywhere equal to $I - \nabla (-\step u_0)^{c_2}$.
\end{proof}
\begin{remark}
We note that \Cref{lem:char-of-solutionmap} describes the distribution one obtains by applying the denoising technique from \cite{lunz2018adversarial} pointwise to an absolutely continuous distribution $\mu$. Indeed, that procedure consists of solving \eqref{prob:CROF} given $x$ when $u_0$ is a Kantorovich potential for $W_1(\mu,\nu)$. It is interesting to observe that while the denoising technique of \cite{lunz2018adversarial} applied to a specific image $x_0$ amounts to an implicit Euler scheme on a Kantorovich potential $u_0$, \Cref{lem:char-of-solutionmap} shows that the distribution one thus obtains on all images is characterised as an explicit Euler step on the functional $W_1(\cdot,\nu)$; this holds since such a $u_0$ is a subgradient of this functional evaluated at $\mu$. Implicit Euler discretizations are often better behaved, motivating us to replace \eqref{prob:forward_euler_inW2} with \eqref{prob:initial-multiscale-statement}. 
 
\end{remark}

\subsection{Equivalence of \eqref{prob:initial-multiscale-statement} and denoising by adversarial regularization}
\label{subsec:optimality-conditions}
Here we will show that that under the assumption that $\step$ is less than the minimal transport length for the ray monotone Wasserstein 1 transport from $\mu$ to $\nu$, the solution to \eqref{prob:initial-multiscale-statement} and the measure obtained via the technique of \cite{lunz2018adversarial} are actually the same. 
\begin{proposition}
\label{lem:multiscale_generalizes_TTC}
Suppose that $\Omega\subset \RR^d$ is compact and convex, and that $\mu,\nu \in \mathcal{P}(\Omega)$. Suppose that $\mu \ll \mathcal{L}_d$, and that $\step>0$ satisfies 
\begin{equation}
\label{eq:minimal_displacement_g_step}
    \essinf_{\mu} |x-T_0(x)| > \step,
\end{equation}
where $T_0$ is the unique ray monotone optimal transport map for $W_1(\mu,\nu)$. Take $u_0 \in \onelip$ a Kantorovich potential for $W_1(\mu,\nu)$, and let $T$ be a Borel map equal to $I-\nabla(-\step u_0)^{c_2}$ $\mu$ almost everywhere. Then $\rho_\step:= T_\# \mu$ is the unique solution to \eqref{prob:initial-multiscale-statement}.
\end{proposition}
\begin{proof}
Since $\Omega$ is convex we immediately obtain that $\lebesgue(\partial \Omega) = 0$ (see, for example, \cite{lang1986note}). Thus, $\mu \ll \lebesgue$ implies that the functional in \eqref{prob:initial-multiscale-statement} is strictly convex (via \cite{santambrogio2015optimal}, Theorem 7.19 again), and so the solution to \eqref{prob:initial-multiscale-statement} is unique if it exists. Next, we claim that if $\rho_0 \in \mathcal{P}(\Omega)$ and there exists $\varphi_0 \in \steplip$ such that
\begin{align}
        \int_\Omega \varphi_0 d \nu - \int_\Omega \varphi_0 d\rho_0 &= \step W_1(\rho_0, \nu) \label{eq:varphi-is-W1-kpot},
\end{align}
and
\begin{align}
        \int_\Omega \varphi_0^{c_2} d\mu + \int_\Omega \varphi_0 d\rho_0 &= \frac{1}{2}W_2^2(\mu,\rho_0),\label{eq:varphi-is-W2-kpot}
\end{align}
then $\rho_0$ solves \eqref{prob:initial-multiscale-statement}. Indeed, by Proposition 7.17 from \cite{santambrogio2015optimal}, assumptions \eqref{eq:varphi-is-W1-kpot} and \eqref{eq:varphi-is-W2-kpot} imply that 
\begin{equation*}
    -\varphi_0 \in \partial \left(\step W_1(\cdot, \nu)\right)(\rho_0), \quad \varphi_0 \in \partial \left(\frac{1}{2}W_2^2(\cdot, \mu)\right)(\rho_0).
\end{equation*}
As such,
\begin{align*}
    0 &= \varphi_0 - \varphi_0,\\
    &\in \partial \left(\frac{1}{2}W_2^2(\cdot, \mu)\right)(\rho_0) + \partial \left(\step W_1(\cdot, \nu)\right)(\rho_0),\\
    &\subset \partial\left(\frac{1}{2}W_2^2(\cdot, \mu) + \step W_1(\cdot, \nu)\right)(\rho_0),
\end{align*}
and thus $\rho_0$ solves \eqref{prob:initial-multiscale-statement}, proving the claim.

Now we assert that these conditions hold for $\rho_\step:= T_\# \mu$ and $\varphi_0:= - \step u_0$. First, we note that by Proposition 9 from \cite{milne2022new} and \Cref{lem:uniquesolvability}, the assumption \eqref{eq:minimal_displacement_g_step} implies that
\begin{equation}
    I-\nabla(-\step u_0)^{c_2}(x) = I-\step \nabla u_0(x)\label{eq:solution_map_is_gradient_descent}
\end{equation}
$\mu$ almost everywhere. Next, observe that convexity of $\Omega$, together with \eqref{eq:minimal_displacement_g_step} and standard properties of Wasserstein 1 Kantorovich potentials, imply that $\rho_\step \in \mathcal{P}(\Omega)$. Also, by Theorem 1 (i) of \cite{milne2021trust}, $u_0$ is a Kantorovich potential for $W_1(\rho_\step, \nu)$. As such, $\varphi_0=-\step u_0$ satisfies $\varphi_0 \in \steplip$ and \eqref{eq:varphi-is-W1-kpot}. Finally, \eqref{eq:varphi-is-W2-kpot} is given by \Cref{lem:subdiff_of_H}, since $\rho_\step = T_\# \mu$, and by definition $T = I-(-\step \nabla u_0)^{c_2}$
$\mu$ almost everywhere.
\end{proof}
The link between \eqref{prob:initial-multiscale-statement} and the denoising method of \cite{lunz2018adversarial} having been established, we now analyse solutions of \eqref{prob:initial-multiscale-statement}. 
\section{A class of minimization problems with solutions given by projections}
\label{sec:deeper_analysis}
In this section we will prove a general theorem about the minimization of a certain class of convex functions, establishing that the solution map is equivalent to a projection. We will show that ROF (see \eqref{prob:ROF}) and \eqref{prob:initial-multiscale-statement} are examples of this class of problems. Thus, we can apply this general theorem to yield \Cref{cor:meyer_for_ROF} and, with additional arguments from optimal transport, our \Cref{prop:meyer-for-multiscale-simplified}. This puts ROF and \eqref{prob:initial-multiscale-statement} within a common framework and provides a fruitful analogy in the sequel. 

Let $X$ be a Hausdorff locally convex topological vector space\footnote{We will not need this amount of generality for our applications, but we phrase our theorem in this setting to indicate that nothing more is needed.}, and take $X^*$ as its continuous dual; in general we will denote by $x$ and $x^*$ points in $X$ and $X^*$ respectively. Let $F:X \rightarrow \RR$ be a proper lower semi-continuous convex functional. Recall that the Legendre dual of such a function is given by $F^* : X^* \rightarrow \RR \cup \{+\infty\}$,
\begin{equation*}
    F^*(x^*):= \sup_{x \in X} \langle x, x^* \rangle - F(x),
\end{equation*}
with $\langle \cdot, \cdot \rangle$ denoting the duality pairing, and set 
\begin{equation*}
\text{dom}(F^*) = \{ x^* \in X^* \mid F^*(x^*) < +\infty\}.    
\end{equation*}
When studying the subdifferential of $F^*$ we will restrict the dual of $X^*$ to $X \subset X^{**}$, i.e.
\begin{equation*}
    \partial F^*(x^*) := \{ x \in X \mid \forall y^* \in X^*, F^*(y^*) \geq F^*(x^*) + \langle x, y^* - x^*\rangle \}. 
\end{equation*}
We will focus on $F$ which are in fact continuous, and such that $F^*$ is a strictly convex function. Take $K \subset X$ as a closed, convex, non-empty set satisfying $K = -K$ and let $1_K$ denote the indicator function of $K$. For $y_0^* \in X^*$, consider the optimization problem
\begin{equation}
    \min_{x^* \in X^*} F^*(x^*) + 1_{K}^*(x^* - y_0^*).\label{prob:general_minimization_problem}
\end{equation}
To motivate the analysis of such problems, we will now  indicate that both ROF and \eqref{prob:initial-multiscale-statement} are examples. 
\begin{example}[ROF]
\label{ex:how_ROF_is_general}
Take $X = L^2(\RR^2)$, and $F: L^2(\RR^2) \rightarrow \RR$ as
\begin{equation*}
    F(u) = \frac{1}{2}\norm{u}_{L^2(\RR^2)}^2 + \langle f, u\rangle.
\end{equation*}
This functional is obviously continuous and convex. It is a simple exercise to show that its dual is
\begin{equation*}
    F^*(u) = \frac{1}{2}\norm{u-f}_{L^2(\RR^2)}^2,
\end{equation*}
which is strictly convex. Take the set $K$ as
\begin{equation*}
    K:= \{ v \in L^2(\RR^2) \mid \norm{v}_* \leq \step\}.
\end{equation*}
It is clear that $K$ is convex, $K = -K$, and $K$ is closed. It is also not difficult to show that
\begin{equation*}
    1_K^*(u):= \sup_{v \in K}\int_{\RR^2} vu dx =\step \norm{u}_{TV}
\end{equation*}
Thus, we see that ROF (i.e. \eqref{prob:ROF}) is an example of \eqref{prob:general_minimization_problem}, with $y_0^* = 0$.
\end{example}

\begin{example}[WROF]
\label{ex:how_multiscale_is_example}
Assume $\Omega \subset \RR^d$ is compact and convex (and therefore $\lebesgue(\partial \Omega) = 0$). Let $X = C(\Omega)$ with the topology induced by the sup norm. Then $X^* = \mathcal{M}(\Omega)$, the set of finite signed Borel measures on $\Omega$.  Let $\mu \in \mathcal{P}(\Omega)$ with $\mu \ll \lebesgue$, and take $F: C(\Omega) \rightarrow \RR$ as the functional
\begin{equation*}
    F(\varphi) := - \int_\Omega \varphi^{c_2} d\mu.
\end{equation*}
It is shown in the proof of Proposition 7.17 of \cite{santambrogio2015optimal} that $F$ defined in this way is convex and continuous, and that $F^*$ satisfies, for $\rho \in \mathcal{M}(\Omega)$,
\begin{equation*}
    F^*(\rho) = \begin{cases}
    \frac{1}{2}W_2^2(\rho,\mu) &\quad \rho \in \mathcal{P}(\Omega),\\
    +\infty &\quad \text{ else }.
    \end{cases}
\end{equation*}
Further, Proposition 7.19 of \cite{santambrogio2015optimal} proves that $F^*$ is strictly convex when $\mu \ll \lebesgue$. 

Now take $K = \steplip$. This set is convex and closed in $C(\Omega)$, and satisfies $K = -K$. In addition, for $\nu, \rho \in \mathcal{P}(\Omega)$, we have
\begin{align*}
    1_K^*(\rho-\nu) &= \sup_{\varphi \in \steplip} \langle \varphi, \rho - \nu\rangle,\\
    &= \step W_1(\rho,\nu).
\end{align*}
Thus, \eqref{prob:initial-multiscale-statement} is of the form \eqref{prob:general_minimization_problem}.
\end{example}

For our analysis of \eqref{prob:general_minimization_problem}, we find it natural to  
define the divergence $D:\text{dom}(F^*) \times \text{dom}(F^*) \rightarrow \RR \cup \{+\infty\}$ by
\begin{equation}
    D(y^*, x^*):= F^*(y^*) - F^*(x^*) - \sup_{x \in \partial F^*(x^*) \cap K} \langle x, y^* - x^* \rangle.\label{def:general_divergence}
\end{equation}
The following lemma shows that $D$ has properties similar to those of a Bregman divergence.
\begin{lemma}
\label{lem:general_divergence_lemma}
For all $y^*, x^* \in \text{dom}(F^*)$, the functional $D$ satisfies
\begin{equation*}
    D(y^*,x^*) \geq 0.
\end{equation*}
Moreover, if $F^*$ is strictly convex, then $D(y^*, x^*) = 0$ if and only if $\partial F^*(x^*) \cap K \neq \emptyset$ and $y^* = x^*$.
\end{lemma}
\begin{proof}
The claim $D(y^*,x^*) \geq 0$ clearly holds if $\partial F^*(x^*) \cap K = \emptyset$. On the other hand, if $\partial F^*(x^*) \cap K  \neq \emptyset$ the definition of the subdifferential of $F^*$ confirms that $D(y^*, x^*) \geq 0$. Clearly, if $\partial F^*(x^*) \cap K \neq \emptyset$ and $y^* = x^*$ we have $D(y^*, x^*) = 0$.  On the other hand, let $F^*$ be strictly convex. If $D(y^*, x^*) = 0$, then take $\epsilon>0$ and $x_\epsilon \in \partial F^*(x^*) \cap K $ such that
\begin{equation*}
    \sup_{x \in \partial F^*(x^*) \cap K} \langle x, y^* - x^* \rangle - \epsilon \leq \langle x_\epsilon, y^* - x^* \rangle.
\end{equation*}
Since $D(y^*, x^*) = 0$, we therefore obtain
\begin{equation*}
    F^*(y^*) \leq F^*(x^*) + \langle x_\epsilon, y^* - x^* \rangle + \epsilon.
\end{equation*}
Hence, for $t \in [0,1]$,
\begin{align*}
    F^*((1-t)x^* + ty^*) &\leq (1-t) F^*(x^*) + t F^*(y^*),\\
    &\leq (1-t) F^*(x^*) + tF^*(x^*) \\
    &\quad + \langle x_\epsilon, (1-t)x^* + ty^* - x^* \rangle  +t \epsilon,\\
    &\leq F^*((1-t)x^* + ty^*) + t \epsilon.
\end{align*}
Since $\epsilon$ is arbitrary, we obtain that $F^*$ is affine on the segment $[x^*, y^*]$, a contradiction to strict convexity unless $x^* = y^*$. 
\end{proof}
\begin{example}
\label{ex:D_comp_for_ROF}
Let us determine $D$ is in the context of ROF. Recall that in this case, $F^*(u) = \frac{1}{2}\norm{u-f}^2_{L^2(\RR^2)}$. Then $\partial F^*(u)$ is a singleton, given by $\{ u-f\}$. So $D(v, u) = +\infty$ unless $\norm{u-f}_* \leq \step$. In that case,
\begin{align*}
    D(v,u) &= \frac{1}{2}\norm{v-f}^2_{L^2(\RR^2)} - \frac{1}{2}\norm{u-f}^2_{L^2(\RR^2)} - \langle u-f, v-u \rangle,\\
    &= \frac{1}{2}\norm{u-v}^2_{L^2(\RR^2)}.
\end{align*}
The description of $D$ in the context of \eqref{prob:initial-multiscale-statement} will be given in \Cref{sec:divergence_in_WROF}.
\end{example}
For a non-empty convex set $K \subset X$ satisfying $K = -K$, define the semi-norm $\norm{\cdot}_K:X \rightarrow \RR \cup \{+\infty\}$ given by
\begin{equation*}
    \norm{x}_K = \inf \{ t >0 \mid \frac{x}{t} \in K\}.
\end{equation*}
We can now state the main result of this section, which provides conditions under which the solution to \eqref{prob:general_minimization_problem}, if it exists, can be expressed as a projection in the divergence $D$ onto the set of $x^*$ such that $\partial F^*(x^*) \cap K \neq \emptyset$. 

\begin{theorem}
\label{thm:general_meyer}
Suppose that $X$ is a Hausdorff locally convex topological vector space, with $X^*$ as its dual. Assume $F: X \rightarrow \RR$ is continuous and convex, and that its dual $F^*$ is strictly convex. Let $K \subset X$ be a closed, convex, non-empty set satisfying $K = -K$. Suppose that $y_0^* \in \text{dom}(F^*)$, and the problem
\begin{equation}
    \sup_{x \in K} \langle x, y_0^* \rangle - F(x) \label{prob:general_dual},
\end{equation}
has a solution $x_0$. Then
\begin{enumerate}[label=\alph*.]
    \item \eqref{prob:general_minimization_problem} has a unique solution $x_0^*$ given by the single element of $\partial F(x_0)$, \label{item:unique_soln_is_subdiff}
    \item $x_0^*$ is also a solution to \label{item:soln_is_proj}
    \begin{equation}
        \min_{ F^*(x^*) \cap K \neq \emptyset} D(y_0^*, x^*),\label{prob:general_projection}
    \end{equation}
    and,
    \item the values of \eqref{prob:general_minimization_problem}, \eqref{prob:general_dual}, and $F^*(y_0^*) - D(y_0^*, x_0^*)$ coincide.\label{item:values_coincide}
\end{enumerate}
Given \textit{b}, we obtain the following dichotomy:
\begin{enumerate}
    \item If $\partial F^*(y_0^*) \cap K \neq \emptyset$, then $x_0^* = y_0^*$.
    \item Otherwise, $x_0^* \neq y_0^*$, and any solution $x_0$ to \eqref{prob:general_dual} satisfies $x_0 \in\partial F^*(x_0^*)$, $\norm{x_0}_K = 1$, and
    \begin{equation}
        \langle x_0, y_0^* - x_0^* \rangle = 1_K^*(x_0^* - y_0^*). \label{eq:extreme_pair}
    \end{equation}
\end{enumerate}
\end{theorem}
\begin{remark}
Let the solution map to \eqref{prob:general_minimization_problem} as a function of $y_0^*$ be denoted $P$. Then the dichotomy presented in \Cref{thm:general_meyer} confirms that $P(P(y_0^*)) = P(y_0^*)$; i.e.~$P$ is a projection.
\end{remark}
Before proving \Cref{thm:general_meyer}, we show how it yields \Cref{cor:meyer_for_ROF}.
\begin{proof}[\textit{Proof of \Cref{cor:meyer_for_ROF}}]
Recalling \Cref{ex:how_ROF_is_general} and \Cref{ex:D_comp_for_ROF}, we have that \eqref{prob:ROF} is of the form \eqref{prob:general_minimization_problem} for $y_0^* = 0$, and
\begin{equation*}
    F(u) = \frac{1}{2}\norm{u}_{L^2(\RR^2)}^2 + \langle f, u\rangle, F^*(u) = \frac{1}{2}\norm{u-f}_{L^2(\RR^2)}^2
\end{equation*}
\begin{equation*}
    K:= \{ v \in L^2(\RR^2) \mid \norm{v}_* \leq \step\}.
\end{equation*}
\begin{equation*}
    D(v,u)= \begin{cases} \frac{1}{2}\norm{u-v}^2_{L^2(\RR^2)} &\quad \norm{u-f}_* \leq \step,\\
    +\infty &\quad  \norm{u-f}_* > \step
    \end{cases}
\end{equation*}
The problem \eqref{prob:general_dual} therefore takes the form
\begin{equation*}
    \max_{\norm{v}_* \leq \step} -\frac{1}{2}\norm{v}_{L^2(\RR^2)}^2 - \langle v, f \rangle.
\end{equation*}
This problem has a unique solution $\tilde{v}_\step$ since $K$ is convex, non-empty and closed, and the function $v \mapsto \frac{1}{2}\norm{v}_{L^2(\RR^2)}^2 + \langle v, f \rangle$ is continuous, strictly convex, and coercive on $L^2(\RR^2)$. We may therefore apply \Cref{thm:general_meyer} to obtain that \eqref{prob:ROF} has a unique solution given by 
\begin{equation*}
    u_\step = f + \tilde{v}_\step.
\end{equation*}
Given our calculation for $D$ in \Cref{ex:D_comp_for_ROF}, we obtain that $u_\step$ is also a solution of the problem in \eqref{prob:general_projection}, which is
\begin{equation*}
    \min_{\norm{u-f}_* \leq \step} \norm{u}_{L^2(\RR^2)}^2. 
\end{equation*}
Thus, if $\norm{f}_* \leq \step$, it is clear that $u_\step= 0$. On the other hand, $\norm{f}_* >\step$ if and only if $\partial F^*(0) \cap K = \emptyset$. Using \Cref{thm:general_meyer}, we obtain that $\tilde{v}_\step\in K$ satisfies $\norm{\tilde{v}_\step}_K = 1$, and
\begin{equation*}
   \int_{\RR^2} u_\step(f-u_\step)dx = \langle \tilde{v}_\step, -u_\step \rangle = \step \norm{u_\step}_{TV}.
\end{equation*}
Finally, we compute $\norm{v}_K = \norm{v}_*/\step$. As such, $\norm{\tilde{v}_\step}_K = 1$ is equivalent to $\norm{\tilde{v}_\step}_* = \step$, and the proof is complete. 
\end{proof}

\begin{proof}[\textit{Proof of \Cref{thm:general_meyer}}]
We start by proving statement \ref{item:unique_soln_is_subdiff} We note that this result is obtainable using Theorem 2.7.1 of \cite{zalinescu2002convex}, but provide an elementary proof here. Let $x_0$ be a solution to \eqref{prob:general_dual}. Then, equivalently, $x_0$ solves
\begin{equation*}
    \min_{x \in X} F(x) - \langle x, y_0^*\rangle + 1_K(x).
\end{equation*}
Noting that $x \mapsto F(x) - \langle x, y_0^* \rangle$ and $1_K(x)$ are both proper convex functions, and that the former is finite and continuous, we can apply part (iii) of Theorem 2.8.7 from \cite{zalinescu2002convex} to conclude that
\begin{equation}
    \partial \left(F - \langle \cdot, y_0^* \rangle + 1_K\right) = \partial F - \{y_0^*\} + \partial 1_K.
\end{equation}
Since $F^*$ is strictly convex, $\partial F(x_0)$ contains at most one element. Further, since $F$ is convex, proper, and continuous, Theorem 2.4.9 from \cite{zalinescu2002convex} shows that $\partial{F}(x)$ is non-empty for all $x \in X$, and thus $\partial F(x)$ contains a unique element for all $x \in X$. Since $x_0$ is a solution of \eqref{prob:general_dual}, the unique element $x_0^* \in \partial F(x_0)$ satisfies
\begin{equation*}
    0 \in x_0^* - y_0^* + \partial 1_K(x_0).
\end{equation*}
Since $K = -K$, we have $\partial 1_K(-x_0) = -\partial 1_K(x_0)$. Thus,
\begin{equation}
    x_0^* - y_0^* \in \partial 1_K(-x_0).\label{eq:extreme_point}
\end{equation}
Next, since $H:X\rightarrow \RR \cup \{+\infty\}$ is proper, convex, and lower semi-continuous, then for all $x$ such that $H(x)< +\infty$ we have the well known fact that\footnote{See, for example, Theorem 2.4.4 from \cite{zalinescu2002convex} for a proof of this in our setting.}
\begin{equation}
    x^* \in \partial H(x) \Leftrightarrow x \in \partial H^*(x^*) \Leftrightarrow H(x) + H^*(x^*) = \langle x, x^*\rangle.\label{eq:inversion_of_subdiff}
\end{equation}
We apply this to $H=1_K$, which is proper because $K$ is non-empty, convex because $K$ is convex, and lower semi-continuous because $K$ is closed. Thus, \eqref{eq:extreme_point} yields
\begin{equation*}
    -x_0 \in \partial 1_K^*(x_0^* - y_0^*).
\end{equation*}
It is an elementary fact that $\left(\partial H(\cdot - y^*)\right)(x^*) = \partial H(x^* - y^*)$. Recalling that $x_0^* \in \partial F(x_0)$ and using \eqref{eq:inversion_of_subdiff} again, we get
\begin{equation*}
    0 \in \partial F^*(x_0^*) + \partial 1_K^*(x_0^* - y_0^*) \subset \partial \left(F^* + 1_K^*(\cdot - y_0^*) \right)(x_0^*),
\end{equation*}
which confirms that $x_0^*$ is a minimizer of \eqref{prob:general_minimization_problem}. By the assumed strict convexity of $F^*$, $x_0^*$ is the unique minimizer. We have shown that if $x_0$ solves \eqref{prob:general_dual}, then the unique $x_0^* \in \partial F(x_0)$ solves \eqref{prob:general_minimization_problem}, which proves statement \ref{item:unique_soln_is_subdiff}


Statements b and c will be proven together. Regarding the values of \eqref{prob:general_minimization_problem} and \eqref{prob:general_dual}, note that by definition of the Legendre dual, for all $x^* \in X^*$ and $x \in K$,
\begin{align*}
    F^*(x^*) + 1_K^*(x^*-y_0^*) &\geq \langle x, x^* \rangle - F(x) + \langle -x, x^* -y_0^* \rangle,\\
    &= \langle x, y_0^* \rangle - F(x).
\end{align*}
Hence,
\begin{equation*}
    \inf_{x^* \in X^*} F^*(x^*) + 1_K^*(x^*-y_0^*) \geq \sup_{x \in K}\langle x, y_0^* \rangle - F(x).
\end{equation*}
On the other hand, for $x_0$ optimal in \eqref{prob:general_dual} and $x_0^* \in \partial F(x_0)$ optimal in \eqref{prob:general_minimization_problem},  \eqref{eq:inversion_of_subdiff} implies
\begin{align*}
    F^*(x_0^*) + 1_K^*(x_0^*-y_0^*) &= \langle x_0, x_0^* \rangle - F(x_0) + \langle -x_0, x_0^* -y_0^* \rangle,\\
    &= \langle x_0, y_0^* \rangle - F(x_0).
\end{align*}
This establishes that the values of \eqref{prob:general_minimization_problem} and \eqref{prob:general_dual} are the same. 

Next, we turn to \eqref{prob:general_projection}. Invoking \eqref{eq:inversion_of_subdiff} again, and recalling that $\partial F(x)$ contains a unique element for all $x \in X$, we obtain that for each $x \in K$, there exists $x^* \in \partial F(x)$ such that $\partial F^*(x^*) \cap K \neq \emptyset$. As such,
\begin{align*}
    \langle x, y_0^*\rangle - F(x) &= \langle x, y_0^*-x^*\rangle + F^*(x^*),\\
    &=F^*(y_0^*) - \left( F^*(y_0^*)- F^*(x^*)  - \langle x, y^*_0-x^*\rangle \right),\\
    &\leq F^*(y_0^*) - (F^*(y_0^*) - F^*(x^*)- \sup_{z \in \partial F^*(x^*)\cap K}\langle z, y^*_0 - x^*\rangle.
\end{align*}
Thus,
\begin{equation}
    \sup_{x \in K}\langle x, y_0^*\rangle - F(x) \leq F(y_0^*) - \inf_{x^*\in X^*}D(y_0^*, x^*).
\end{equation}
On the other hand, for $x^* \in X^*$ with $\partial F^*(x^*) \cap K \neq \emptyset$, let $\epsilon >0$ and take $x\in \partial F^*(x^*) \cap K$ satisfying
\begin{equation*}
    \sup_{z \in \partial F^*(x^*)\cap K}\langle z, y^*_0 - x^*\rangle - \epsilon \leq \langle x, y^*_0 - x^* \rangle.
\end{equation*}
Then
\begin{align*}
    F(y_0^*) - D(y_0^*, x^*) &\leq F(y_0^*) - (F^*(y_0^*) - F^*(x^*)- \langle x, y^*_0 - x^*\rangle - \epsilon),\\
    &= \langle x, y_0^* \rangle - F(x) +\epsilon.
\end{align*}
Hence,
\begin{equation*}
     \sup_{x \in K}\langle x, y_0^*\rangle - F(x) + \epsilon \geq F(y_0^*) - \inf_{x^*\in X^*}D(y_0^*, x^*).
\end{equation*}
Since $\epsilon$ is arbitrary, we obtain equality of the values of $F(y_0^*) - \inf_{x^*}D(y_0^*, x^*)$ and \eqref{prob:general_dual}. Finally, if $x_0$ solves \eqref{prob:general_dual}, then we know that $x_0^* \in \partial F(x_0)$ solves \eqref{prob:general_minimization_problem}. We also have
\begin{align*}
    \sup_{x\in K}\langle x, y_0^* \rangle - F(x) &= \langle x_0, y_0^* \rangle - F(x_0),\\
    &\leq F^*(y_0^*) - D(y_0^*, x_0^*),\\
    &\leq F^*(y_0^*) - \inf_{x^*\in \text{dom}(F^*)} D(y_0^*, x^*),\\
    &=  \sup_{x\in K}\langle x, y_0^* \rangle - F(x).
\end{align*}
Equality of the first expression and the last mean that each inequality is an equality; thus $x_0^*$ solves \eqref{prob:general_projection} as claimed, which completes the proof of statements b and c.

We now address the dichotomy. If $\partial F^*(y_0^*) \cap K \neq \emptyset$, then by \Cref{lem:general_divergence_lemma} the only minimizer of \eqref{prob:general_projection} is $y_0^*$. So suppose $\partial F^*(y_0^*) \cap K = \emptyset$. Then it is clear that $x_0^* \neq y_0^*$, since $D(y_0^*, x_0^*) < +\infty$, and hence $\partial F^*(x_0^*) \cap K \neq \emptyset$. For any solution $x_0$ to \eqref{prob:general_dual}, we obtain $x_0 \in \partial F^*(x_0^*)$ by \eqref{eq:inversion_of_subdiff}. We also have \eqref{eq:extreme_point}, and thus via the last equality of \eqref{eq:inversion_of_subdiff},
\begin{equation*}
    \langle -x_0, x_0^* - y_0^* \rangle = 1_K^*(x_0^* - y_0^*),
\end{equation*}
which is \eqref{eq:extreme_pair}. Since $x_0 \in K$, we have $\norm{x_0}_K \leq 1$. On the other hand, if $\norm{x_0}_K <1$, then since $x_0^* \neq y_0^*$ there is no possibility of \eqref{eq:extreme_pair} holding.
\end{proof}

\section{Proof of \Cref{prop:meyer-for-multiscale-simplified}}
\label{sec:applying_general_to_WROF}
In \Cref{sec:prelim-for-proving-meyer} we will demonstrate that the hypotheses of \Cref{thm:general_meyer} hold for \eqref{prob:initial-multiscale-statement}. In addition, we will describe the divergence $D$ in this context. In \Cref{sec:applying_general_meyer_to_multiscale} we will use these preliminaries to complete the proof of \Cref{prop:meyer-for-multiscale-simplified}.
\subsection{Preliminaries}
\label{sec:prelim-for-proving-meyer}
Recall that in the context of \eqref{prob:initial-multiscale-statement}, $K = \steplip$, and $F: C(\Omega) \rightarrow \RR$ given by
\begin{equation*}
    F(\varphi) = - \int_\Omega \varphi^{c_2} d\mu.
\end{equation*}
We mentioned in \Cref{ex:how_multiscale_is_example} that $F$ defined in this way is convex and continuous, and that $F^*$ is strictly convex provided $\mu \ll \lebesgue$ and $\lebesgue(\partial \Omega) = 0$. Further, $K = \steplip$ is closed, convex, and non-empty, and satisfies $K = -K$. The only remaining hypothesis of \Cref{thm:general_meyer} to verify is that \eqref{prob:general_dual} has a solution. In this setting \eqref{prob:general_dual} takes the form
\begin{equation}
\sup_{\varphi \in \steplip} \int_\Omega \varphi^{c_2} d\mu + \int_\Omega \varphi d\nu.\label{prob:firstdual}
\end{equation}
The existence of a solution could be proved by standard arguments, but we will do so by re-writing \eqref{prob:firstdual} as an unconstrained problem in terms of the Huber cost function $\hybrid$ (see \eqref{eq:hybridcost} for the definition); this will be useful to us later.

For $\mu, \rho \in \mathcal{P}(\Omega)$, let $\cI_{\hybrid}(\mu,\rho)$ be the transport cost, i.e.
\begin{equation*}
    \cI_{\hybrid}(\mu,\rho) := \inf_{\gamma \in \Pi(\mu,\rho)}\int_{\Omega \times \Omega}\hybrid(x,y) d\gamma(x,y),
\end{equation*}
where $\Pi(\mu,\rho)$ is the set of probability distributions on $\Omega \times \Omega$ with marginal distributions given by $\mu$ and $\rho$. Since $c_2(x,y) \geq \hybrid(x,y)$, we have $\squaredw{\mu}{\rho}  \geq \cI_{\hybrid}(\mu,\rho)$. We now prove an easy fact about the $\hybrid$-transform.
\begin{lemma}
\label{lem:equality_of_c_transforms}
 If $\Omega$ is convex and $\varphi \in \steplip$, then 
 \begin{equation}
 \varphi^{\hybrid}(x) = \varphi^{c_2}(x) \quad \forall x \in \Omega \label{eq:equality_of_c_transforms}
 \end{equation}
 where, recall,
 \begin{equation*}
     \varphi^{\hybrid}(x) := \inf_{y \in \Omega}\hybrid(x,y) - \varphi(y).
 \end{equation*}
\end{lemma}
\begin{proof}
Let $B_\step(x)$ be the closed Euclidean ball of radius $\step$ centred at $x$. We claim that $\varphi\in \steplip$ implies that for all $x \in \Omega$, 
\begin{equation}
\varphi^{c_2}(x) = \inf_{y \in \Omega \cap B_\step(x)} c_2(x,y) - \varphi(y). \label{eq:infforc2takenovertauball}
\end{equation}
Indeed, for $y \in \Omega \setminus B_\step(x)$, let $z$ be the projection of $y$ onto $B_\step(x)$ in the Euclidean norm; note that $z \in \Omega$ by convexity of $\Omega$. Since $c_2$ is convex and differentiable in its second variable, we have
\begin{equation*}
c_2(x, y) \geq c_2(x,z) + \langle z-x, y-z \rangle = c_2(x,z) + \step |z-y|,
\end{equation*}
so
\begin{equation*}
c_2(x,y) -\varphi(y) \geq c_2(x,z) + \step |z-y| - \varphi(y) \geq c_2(x,z) - \varphi(z).
\end{equation*}
This proves \eqref{eq:infforc2takenovertauball}. We can also prove, by a nearly identical argument, that the infimum in 
\begin{align*}
\varphi^{\hybrid}(x) &= \inf_{y\in \Omega} \hybrid(x,y) - \varphi(y),
\end{align*}
can also be restricted to $B_\step(x)\cap \Omega$.  Since $\hybrid(x,y) = c_2(x,y)$ when $|x-y| \leq \step$, the conclusion follows.
\end{proof}
An immediate consequence of the preceding lemma is that we can re-write \eqref{prob:firstdual} as an unconstrained problem in terms of the cost $\hybrid$. 
\begin{lemma}
\label{lem:first_dual_and_second_are_equivalent}
When $\Omega$ is convex, the problems \eqref{prob:firstdual} and 
\begin{equation}
\sup_{\varphi \in C(\Omega)} \int_\Omega \varphi^{\hybrid} d\mu + \int_\Omega \varphi d\nu\label{prob:seconddual}
\end{equation}
are equivalent; they have the same value, a solution to \eqref{prob:firstdual} is a solution to \eqref{prob:seconddual}, and a $\hybrid$-concave solution to \eqref{prob:seconddual} is a solution to \eqref{prob:firstdual}. 
\end{lemma}
\begin{proof}
Let $\varphi \in C(\Omega)$ be a candidate for maximizing \eqref{prob:seconddual}. Without loss of generality we may take $\varphi$ $\hybrid$-concave, and since $\hybrid(x,y) = h(|x-y|)$ for $h\in \step\text{-Lip}(\RR^+)$, we obtain $\varphi \in \steplip$ as well. So \eqref{prob:seconddual} can be re-written as
\begin{equation*}
    \sup_{\varphi \in \steplip} \int_\Omega \varphi^{\hybrid} d\mu + \int_\Omega \varphi d\nu
\end{equation*}
We have already shown in \Cref{lem:equality_of_c_transforms} that when $\Omega$ is convex and $\varphi \in \steplip$,  $\varphi^{\hybrid} = \varphi^{c_2}$. This establishes the equivalence of the problems. 
\end{proof}
The existence of a solution to \eqref{prob:firstdual} now follows from the existence of a Kantorovich potential for the transport problem $\cI_{\hybrid}(\mu,\nu)$.

\begin{lemma}
\label{lem:solution_to_first_dual_exists}
Let $\Omega\subset \RR^n$ be compact and convex. For $\mu, \nu \in \mathcal{P}(\Omega)$, problem \eqref{prob:firstdual} has a solution.
\end{lemma}

\begin{proof}
Since the cost $\hybrid$ is uniformly continuous and bounded on $\Omega \times \Omega$, we may use Theorem 1.39 of \cite{santambrogio2015optimal} to conclude that there exists a $\hybrid$-concave function $\varphi_\step$ such that
\begin{align*}
    \cI_{\hybrid}(\mu,\nu) &= \sup_{\varphi \in C(\Omega)} \int_\Omega \varphi^{\hybrid} d\mu + \int_\Omega \varphi d\nu,\\
    &=\int_\Omega \varphi_\step^{\hybrid} d\mu + \int_\Omega \varphi_\step d\nu.
\end{align*}
By \Cref{lem:first_dual_and_second_are_equivalent}, $\varphi_\step$ is a solution of \eqref{prob:firstdual}.
\end{proof}

The hypotheses of \Cref{thm:general_meyer} being validated, we may apply it to \eqref{prob:initial-multiscale-statement}, and we will do so in \Cref{sec:applying_general_meyer_to_multiscale}. As a preliminary step, however, it will be helpful to specify the subdifferential of $F$, since the minimizer of \eqref{prob:initial-multiscale-statement} will be given by $\partial F(\varphi_\step)$ if $\varphi_\step$ solves \eqref{prob:firstdual}.  
\begin{lemma}
\label{lem:subdiff_of_H}
For $\varphi \in C(\Omega)$, $\partial F(\varphi)$ is non-empty, and 
\begin{equation}
\partial F(\varphi) = \{ \rho \in \mathcal{P}(\Omega) \mid \int_\Omega \varphi^{c_2} d\mu + \int_\Omega \varphi d\rho = \frac{1}{2}W_2^2(\rho,\mu)\}. \label{eq:subdiff_of_H}
\end{equation}
Further, if $\partial \Omega$ has Lebesgue measure $0$, $\mu \ll \lebesgue$, and $T:\Omega\rightarrow\Omega$ is any Borel map $\mu$ almost everywhere equal to $I-\nabla \varphi^{c_2}$, then
\begin{equation*}
    \partial F(\varphi) = \{ T_\# \mu\}.
\end{equation*}
\end{lemma}
\begin{proof}
Since $F$ is convex, proper, and continuous everywhere, Theorem 2.4.9 from \cite{zalinescu2002convex} shows that $\partial{F}(\varphi)$ is non-empty for all $\varphi$. Since $F$ is a convex, proper, and continuous function, we invoke \eqref{eq:inversion_of_subdiff} to state that
\begin{equation*}
    \rho \in \partial F(\varphi) \Leftrightarrow \varphi \in \partial F^*(\rho) = \partial \left(\frac{1}{2}W_2^2(\cdot,\mu)\right)(\rho).
\end{equation*}
Via Proposition 7.17 from \cite{santambrogio2015optimal}, we obtain \eqref{eq:subdiff_of_H}. 

Thus, $\rho \in \partial F(\varphi)$ means that $\varphi^{c_2}$ is a Kantorovich potential for $W_2(\mu,\rho)$. Suppose in addition $\partial \Omega$ has Lebesgue measure $0$ and $\mu \ll \lebesgue$.  The characterisation of the optimal transport map for the cost $c_2(x,y) = \frac{1}{2}|x-y|^2$ in Theorem 1.17 of \cite{santambrogio2015optimal} then confirms that $\rho = T_\# \mu$, for any $T$ $\mu$ almost everywhere equal to $I-\nabla \varphi^{c_2}$.
\end{proof}
It will also be useful to study the divergence $D$ in the context of \eqref{prob:initial-multiscale-statement}, specifically where it is finite. This is the content of the next subsection.

\subsubsection{The divergence $D$ in the context of \eqref{prob:initial-multiscale-statement}}
\label{sec:divergence_in_WROF}
Here we will provide a characterisation of the set of measures $\rho$ such that $D(\nu, \rho)<+\infty$. We will also provide an economic interpretation of $D$ on this set.

First, set $B_\step(\mu)$ as the set of all measures $\rho$ that are reachable from $\mu$ under an optimal plan for the cost $\hybrid$ such that no point moves more than distance $\step$,
\begin{equation}
    B_\step(\mu)= \{  \rho \in \mathcal{P}(\Omega) \mid \exists \gamma_0 \text{ optimal for } \cI_{\hybrid}(\mu,\rho) \text{ s.t.} \spt(\gamma_0) \subset \{|x-y|\leq \step\}\}.\label{eq:ball_def}
\end{equation}
We consider $B_\step(\mu)$ because the following lemma shows that it is exactly the set of $\rho \in \mathcal{P}(\Omega)$ such that $\partial F^*(\rho) \cap K \neq \emptyset$, and thus $D(\nu,\rho)< +\infty$. In particular it is the set of measures $\rho$ such that $W_2^2(\rho, \mu)$ has an $\step$-Lipschitz Kantorovich potential. We also provide a third characterisation of $B_\step(\mu)$ as the set of all measures which are close enough to $\mu$ that there are no savings to be had using the discounted cost $\hybrid$. 

\begin{lemma}
\label{lem:alternative_char_of_Beta}
Let $\Omega$ be compact and convex. Then the following are equivalent
\begin{enumerate}
    \item $\rho \in B_\step(\mu)$,
    \item $\partial \left(\frac{1}{2}W_2^2(\cdot, \mu)\right)(\rho) \cap \steplip \neq \emptyset$, and,
    \item $\squaredw{\mu}{\rho} = \cI_{\hybrid}(\mu,\rho)$.
\end{enumerate}
\end{lemma}
\begin{proof}
We will proceed by proving $1 \Rightarrow 2 \Rightarrow 3 \Rightarrow 1$. Let $\rho \in B_\step(\mu)$.  Let $\varphi$ be a $\hybrid$-concave function such that
\begin{equation*}
    \cI_{\hybrid}(\mu,\rho) = \int_\Omega \varphi^{\hybrid} d\mu + \int_\Omega \varphi d\rho.
\end{equation*}
Since $\varphi$ is $\hybrid$-concave we obtain that $\varphi \in \steplip$. For $\gamma_0$ the optimal plan transporting $\mu$ to $\rho$ from the definition of $B_\step(\mu)$, we have
\begin{align*}
    \squaredw{\mu}{\rho} &\leq \frac{1}{2}\int_{\Omega \times \Omega} |x-y|^2 d\gamma_0,\\
    &= \int_{\Omega \times \Omega} \hybrid(x,y) d\gamma_0,\\
    &= \int_\Omega \varphi^{\hybrid} d\mu + \int_\Omega \varphi d\rho,\\
    &= \int_\Omega \varphi^{c_2}d\mu + \int_\Omega \varphi d\rho,\\
    &\leq \squaredw{\mu}{\rho}.
\end{align*}
In the second to last line we have used \Cref{lem:equality_of_c_transforms}. Equality of the first and last terms mean we have equality throughout, and thus $\varphi \in \partial \left(\frac{1}{2}W_2^2(\cdot, \mu)\right)(\rho) \cap \steplip$. 

Second, if $\rho \in \mathcal{P}(\Omega)$ is such that there exists $\varphi$ satisfying
\begin{equation*}
    \varphi \in \partial \left(\frac{1}{2}W_2^2(\cdot, \mu)\right)(\rho) \cap \steplip,
\end{equation*}
then, using \Cref{lem:equality_of_c_transforms},
\begin{align*}
    \squaredw{\mu}{\rho}&=  \int_\Omega \varphi^{\hybrid}d\mu + \int_\Omega \varphi d\rho,\\
    &\leq \cI_{\hybrid}(\mu,\rho).
\end{align*}
Since $\squaredw{\mu}{\rho} \geq \cI_{\hybrid}(\mu,\rho)$ in general, we have $\squaredw{\mu}{\rho} = \cI_{\hybrid}(\mu,\rho)$. 

Finally, suppose $\squaredw{\mu}{\rho} = \cI_{\hybrid}(\mu,\rho)$. Since $\Omega$ is compact, there exists an optimal plan $\otplan\in \Pi(\mu,\rho)$ for $W_2(\mu, \rho)$. We compute
\begin{align}
    \squaredw{\mu}{\rho} &= \int_{\Omega \times \Omega} c_2(x,y) d\otplan(x,y),\nonumber\\
    &\geq \int_{\Omega \times \Omega} \hybrid(x,y) d\otplan(x,y),\label{eq:diff_is_where_larger_than_eta}\\
    &\geq \cI_{\hybrid}(\mu,\rho),\nonumber\\
    &= \squaredw{\mu}{\rho}.\nonumber
\end{align}
Equality of the first and last term means we have equality throughout. This indicates that $\gamma_0$ is optimal for $\cI_{\hybrid}(\mu,\rho)$, and
\begin{equation*}
    \otplan(\{ (x,y) \in \Omega \mid |x-y| > \step\}) = 0,
\end{equation*}
otherwise the inequality in \eqref{eq:diff_is_where_larger_than_eta} would be strict. Thus, $\rho \in B_\step(\mu)$.
\end{proof}


In the  context of \eqref{prob:initial-multiscale-statement}, the divergence $D$ previously defined in \eqref{def:general_divergence} takes the following form:
\begin{align}
    \bregmandiv{}(\nu, \rho) &= \frac{1}{2}W_2^2(\nu, \mu) - \frac{1}{2}W_2^2(\rho, \mu) \nonumber\\
    &\quad -  \sup \{ \langle \varphi, \nu - \rho \rangle \mid \varphi \in \partial \left(\frac{1}{2}W_2^2(\cdot, \mu)\right)(\rho) \cap \steplip\}, \label{def:D_deff}
\end{align}
with $\mu\in \mathcal{P}(\Omega)$ a fixed reference measure. Here we have introduced the notation $D_\step$ to make the dependence of $D$ on the scale $\step$ explicit. 

We now detail the economic interpretation of \eqref{def:D_deff} that we mentioned in \Cref{sec:intro-characterisation-of-solution}. Here we assume that goods are manufactured with distribution $\mu$, purchased from the manufacturer with distribution $\rho$ and sold to consumers with distribution $\nu$. In this setting $D_\step(\nu,\rho)$ represents the total loss of value in a supply chain when the transport cost has an economy of scale and consumers adopt a ``buy local'' policy.

Indeed if anyone can move goods from $x$ to $y$ for a transport cost of $\hybrid(x,y)$, it is well known\footnote{See, for example, \cite{villani2009optimal}, page 65} that the maximum profit obtainable for transporting $\mu$ to $\rho$ while still being competitive with this global shipping rate  is $\cI_{\hybrid}(\mu,\rho)$, and that a  potential
\begin{equation*}
 \varphi \in \partial \left(\cI_{\hybrid}(\cdot, \mu)\right)(\rho)   
\end{equation*} represents an optimal sale price as a function of location. We suppose that instead of shipping directly to consumers, the manufacturer sells to a retailer, who purchases product with distribution $\rho$ and sells with distribution $\nu$, both at price $\varphi$. The profits obtained by the retailer are therefore
\begin{equation*}
    \langle \varphi, \nu - \rho\rangle.
\end{equation*}
Given $\mu$ and $\rho$, there may be several optimal prices $\varphi$, and since all of them result in the same benefit for the manufacturer, they allow the retailer to choose one that maximizes their profit. However, the manufacturer specifies that $\varphi \in \steplip$; otherwise the retailer may be able to exploit an arbitrage against the global shipping cost $\hybrid$. The profits of the retailer are then
\begin{equation*}
    \sup \{\langle \varphi, \nu - \rho\rangle \mid  \varphi \in \partial \left(\cI_{\hybrid}(\cdot, \mu)\right)(\rho)\cap \steplip\}.
\end{equation*}
We now suppose that the consumers impose a ``buy local'' policy, in the sense that they will not tolerate goods being shipped more than distance $\step$ to retailers. The retailer must modify $\rho$ to compensate for this, and by definition the only admissible distributions are those in $B_\step(\mu)$. If $\rho \in B_\step(\mu)$ however, \Cref{lem:equality_of_c_transforms} and \Cref{lem:alternative_char_of_Beta} show that   
\begin{equation*}
    \varphi \in \partial \left(\cI_{\hybrid}(\cdot, \mu)\right)(\rho)\cap \steplip \Leftrightarrow  \varphi \in \partial \left(\squaredw{\cdot}{\mu}\right)(\rho) \cap \steplip.
\end{equation*}
Since $\cI_{\hybrid}(\mu,\rho) = \squaredw{\mu}{\rho}$ for $\rho \in B_\step(\mu)$, the total profits for both manufacturer and retailer are
\begin{equation*}
    \squaredw{\mu}{\rho} + \sup \{\langle \varphi, \nu - \rho\rangle \mid  \varphi \in \partial\left( \squaredw{\cdot}{\mu}\right)(\rho)\cap \steplip\}.
\end{equation*}
Subtracting this from the baseline $\squaredw{\mu}{\nu}$, we see that $D_\step(\nu,\rho)$ is indeed the total loss of value when product is purchased by retailers at distribution $\rho$ and sold at distribution $\nu$ under a buy local policy for consumers and when transportation over scale $\step$ is discounted.

\subsection{Applying \Cref{thm:general_meyer} to  (WROF)}
\label{sec:applying_general_meyer_to_multiscale}

With $B_\step(\mu)$ and $\bregmandiv{}$ defined, we can finally apply \Cref{thm:general_meyer} to characterise $\rho_\step$, the unique minimizer of \eqref{prob:initial-multiscale-statement}, as a projection of $\nu$ onto $B_\step(\mu)$ with respect to the divergence $\bregmandiv{}$. The following result is a more detailed version of \Cref{prop:meyer-for-multiscale-simplified} from \Cref{sec:multiscale_introduction}.

\begin{theorem}
    
\label{prop:meyer-for-multiscale}
Let $\Omega$ be compact and convex with non-negligible interior, and suppose $\mu \ll \mathcal{L}_d$. Then
\begin{enumerate}[label=\alph*.]
    \item \eqref{prob:initial-multiscale-statement} has a unique solution $\rho_\step = (T_\step)_\#\mu$, where $T_\step = I- \nabla \varphi^{c_2}_\step$ almost everywhere and $\varphi_\step$ solves \eqref{prob:firstdual}
    \item $\rho_\step$ is also a solution to
    \begin{equation}
        \min_{\rho \in B_\step(\mu)}\bregmandiv{}(\nu,\rho)\label{prob:projection_problem}
    \end{equation}
    and
    \item the values of \eqref{prob:initial-multiscale-statement}, \eqref{prob:firstdual}, and $\squaredw{\nu}{\mu} - \bregmandiv{}(\nu,\rho_\step)$ coincide.
\end{enumerate}
Given statement b, we have the following dichotomy.
\begin{enumerate}
 \item If $\nu \in B_\step(\mu)$, then $\rho_\step= \nu$.
 \item Otherwise, $\rho_\step \neq \nu$. Furthermore, any solution $\varphi_\step$ to \eqref{prob:firstdual} satisfies $\varphi_\step \in \partial\left( \frac{1}{2}W_2^2(\cdot, \mu)\right)(\rho_\step)$, $\text{Lip}(\varphi_\step)= \step$, and
\begin{equation}
    \langle  \varphi_\step,  \nu - \rho_\step \rangle = \step W_1(\rho_\step, \nu).\label{eq:extremality_relation}
\end{equation}
\end{enumerate}
Finally, $T_\step$ is the unique optimal transport map for $W_2(\mu,\rho_\step)$, and satisfies \eqref{eq:displacement_bound_for_w2optimal}.
\end{theorem}
\begin{remark}
This result, together with \Cref{lem:first_dual_and_second_are_equivalent}, provides a proof of statement 1 of \Cref{prop:relation_between_Tstep_andvarphistep}. Moreover, recalling \Cref{lem:uniquesolvability}, we observe that $I-\nabla \varphi^{c_2}_0$ is the solution map to \eqref{eq:different_learned_reg}. Thus, \Cref{prop:meyer-for-multiscale} also proves statement 2 of \Cref{prop:relation_between_Tstep_andvarphistep}.
\end{remark}
\begin{proof}
We have already described how $F(\varphi) = -\int_\Omega \varphi^{c_2} d\mu$ and $K = \steplip$ satisfy the hypotheses of \Cref{thm:general_meyer}; in particular, $\mu \ll \lebesgue$ and $\lebesgue(\partial \Omega)=0$ guarantee strict convexity of $F^*$. Further, \Cref{lem:solution_to_first_dual_exists} guarantees the existence of a solution to \eqref{prob:firstdual}. Statements a, b, and c then follow immediately from \Cref{thm:general_meyer} and \Cref{lem:subdiff_of_H}. 

Since we have shown that $\rho \in B_\step(\mu)$ is equivalent to $\partial F^*(\rho) \cap K \neq \emptyset$ in \Cref{lem:alternative_char_of_Beta}, we see that the condition of the dichotomies in this proposition and \Cref{thm:general_meyer} correspond. The only part of statements 1 and 2 in \Cref{prop:meyer-for-multiscale} that is not an immediate implication of \Cref{thm:general_meyer} is that $\text{Lip}(\varphi_\step) = \step$, but this comes from determining that $\norm{\varphi_\step}_K = \text{Lip}(\varphi_\step)/\step$. Further, $T_\step$ is optimal for $W_2(\mu,\rho_\step)$ since $T_\step = I-\nabla \varphi_\step^{c_2}$ almost everywhere and $\varphi_\step \in \partial\left( \frac{1}{2}W_2^2(\cdot, \mu)\right)(\rho_\step)$. Finally, \eqref{eq:displacement_bound_for_w2optimal} holds since $\text{Lip}(\varphi_\step^{c_2}) \leq \lambda$ by \Cref{lem:equality_of_c_transforms}.
\end{proof}

This result, together with \Cref{lem:first_dual_and_second_are_equivalent} furnishes an additional description of the value of \eqref{prob:initial-multiscale-statement} which is useful in proving the interpretation of $D_\step(\nu, \rho_\step)$ in \eqref{eq:D_interp}.
\begin{corollary}
\label{cor:wrof-is-hybrid}
Under the hypotheses of \Cref{prop:meyer-for-multiscale}, the minimal value of \eqref{prob:initial-multiscale-statement} is equal to 
\begin{equation}
\cI_{\hybrid}(\mu,\nu) = \inf \{ \int_{\Omega \times \Omega} \hybrid(x,y) d\gamma \mid \gamma \in \Pi(\mu,\nu)\}. \label{prob:underlyingOTproblem}
\end{equation}
\end{corollary}
\begin{proof}
    Observe that \eqref{prob:seconddual} is a standard Kantorovich potential problem, and thus via \Cref{prop:meyer-for-multiscale}, \Cref{lem:first_dual_and_second_are_equivalent}, and Theorem 1.39 of \cite{santambrogio2015optimal}  we get that the value of \eqref{prob:initial-multiscale-statement} coincides with \eqref{prob:underlyingOTproblem}.

\end{proof}

We now turn to the proof of \Cref{thm:analogue_of_wavelet_shrinkage}. A crucial role is played by the absolute continuity of $\rho_\step$, and the proof of this property is the focus of the following section.

\section{Absolute continuity of $\rho_\step$}
\label{sec:absolute_continuity_of_rho0}
The following proposition provides conditions under which $\rho_\step$ is guaranteed to be absolutely continuous, and proves statement \ref{claim:rho_step_AC} from \Cref{thm:analogue_of_wavelet_shrinkage}.
\begin{proposition}
Suppose that $\Omega \subset \RR^d$ is compact and convex, with a non-empty interior. If $\mu$ and $\nu$ are absolutely continuous with respect to Lebesgue measure, then $\rho_\step$, the unique solution to \eqref{prob:initial-multiscale-statement}, is absolutely continuous as well.\label{thm:rho0AC}
\end{proposition}
The rest of this section is devoted to proving \Cref{thm:rho0AC}, and the plan is as follows. First, in \Cref{subsec:alternative_form_of_first_dual} we use the alternate expression for the dual problem \eqref{prob:firstdual} furnished by \eqref{prob:seconddual} to obtain a better understanding of how $\rho_\step$ relates to $\mu$ and $\nu$. Namely, there is an optimal transport plan $\gamma_0$ for \eqref{prob:underlyingOTproblem}, and $\rho_\step$ is obtained by completing all transport in this plan that moves less than distance $\step$, as well as progressing all transport that moves more than distance $\step$ as much as possible while retaining $\rho_\step \in B_\step(\mu)$. We use this understanding to decompose $\rho_\step$ into a sum of two measures, and by proving that each of these is absolutely continuous, we will obtain that $\rho_\step$ is absolutely continuous as well. 
\subsection{Consequences of \Cref{lem:first_dual_and_second_are_equivalent} for a minimizer of \eqref{prob:initial-multiscale-statement}}
\label{subsec:alternative_form_of_first_dual}
Recall \Cref{cor:wrof-is-hybrid}, which says that the value of \eqref{prob:initial-multiscale-statement} coincides with that of \eqref{prob:underlyingOTproblem}. By Theorem 1.4 of \cite{santambrogio2015optimal}, an optimal plan for the latter exists since $\Omega$ is compact; throughout this section we will refer to this plan by the notation $\gamma_0$. Let us also  fix $\varphi_\step$ as a solution of \eqref{prob:firstdual} which is $\hybrid$-concave; such a $\varphi_\step$ exists by \Cref{lem:first_dual_and_second_are_equivalent}. The following simple result characterises $\nabla \varphi_\step^{c_2}$, and thus the solution of \eqref{prob:initial-multiscale-statement} (see \Cref{prop:meyer-for-multiscale}), in terms of $\gamma_0$.

\begin{lemma}
\label{lem:gradientofvarphi}
Let $\Omega$ be compact and convex with non-negligible interior. Let $\gamma_0$ be optimal in \eqref{prob:underlyingOTproblem}. If $(x,y) \in \spt(\gamma_0)$, with $x$ in the interior of $\Omega$ and a differentiable point of $\varphi^{c_2}_\step$, then
\begin{equation*}
\nabla \varphi_\step^{c_2}(x) = \begin{cases}
x - y &\quad |x-y| \leq \step,\\
\step \frac{x-y}{|x-y|} &\quad |x-y| \geq \step.
\end{cases}
\end{equation*}
Thus, there is at most one $y \in B_\step(x)$ such that $(x,y) \in\spt(\gamma_0)$ and in that case $x - \nabla \varphi_0^{c_2}(x) = y$. 
\end{lemma}
\begin{proof}
Since $\varphi_\step$ solves \eqref{prob:firstdual},  \Cref{lem:first_dual_and_second_are_equivalent} implies that $\varphi_\step$ also solves \eqref{prob:seconddual}, and \Cref{lem:equality_of_c_transforms} gives that $\varphi^{c_2} = \varphi^{\hybrid}$. Since $\varphi_\step$ is optimal potential in \eqref{prob:seconddual} we have
\begin{equation*}
\varphi^{c_2}_\step(x) + \varphi_\step(y) \leq \hybrid(x,y)
\end{equation*}
with equality on the support of $\gamma_0$. Thus, if $(x,y) \in \spt(\gamma_0)$, the minimum of
\begin{equation*}
\inf_z \hybrid(z,y) - \varphi^{c_2}_\step(z)
\end{equation*}
is obtained at $x$. If $x$ is interior to $\Omega$ and a differentiable point of $\varphi^{c_2}_\step$, then
\begin{equation*}
0 = \nabla_x \hybrid(x,y) - \nabla \varphi_\step^{c_2}(x).
\end{equation*}
Computing the derivative of $\hybrid$, we obtain the claim.
\end{proof}
We note that since $T_\step(x) = x - \nabla \varphi_\step^{c_2}(x)$ almost everywhere, \Cref{lem:gradientofvarphi} proves statement 3 of \Cref{prop:relation_between_Tstep_andvarphistep}.
\subsection{A decomposition of $\rho_\step$}
\label{sec:adecompof_rho}
Define the Borel measures
\begin{equation}
\muc = (\pi_x)_\# \gamma_0 |_{\{|x-y| \leq \step\}}, \quad \mud = (\pi_x)_\# \gamma_0 |_{\{|x-y| > \step\}},\label{def:muc_and_mud}
\end{equation}
where $\pi_x$ and $\pi_y$ are the canonical projections. Let $T_\step$ be a Borel map which is almost everywhere equal to $I-\nabla \varphi_\step^{c_2}$. Recalling that $T_\step$ is optimal for the transport between $\mu$ and $\rho_\step$ for the cost $c_2$ (see \Cref{prop:meyer-for-multiscale}), define
\begin{equation}
\rhonc = (T_\step)_\# \muc, \quad \rhond = (T_\step)_\# \mud.\label{def:rhonc_and_rhond}
\end{equation}
It is clear that $\mu = \muc + \mud$, and from this we obtain $\rho_\step = \rhonc + \rhond$. We will prove that $\rho_\step \ll \mathcal{L}_d$ by showing the same for $\rhonc$ and $\rhond$. 

It is easier to prove that $\rhonc \ll \mathcal{L}_d$, and that is the content of the following lemma. We will actually prove the stronger result that $\rhonc(E) \leq \nu(E)$ for all Borel $E$. This inequality should be expected given the discussion following \Cref{lem:gradientofvarphi}, which says that the map $I-\nabla \varphi_\step^{c_2}$ completes all transport in $\gamma_0$ that moves less than distance $\step$. Since the mass that moves less than distance $\step$ under $\gamma_0$ is precisely $\muc$, and $\gamma_0$ transports $\mu$ to $\nu$, that $\rhonc \leq \nu$ is not surprising. 

\begin{lemma}
\label{lem:rho_0_bounded_above}
If $\mu\ll \mathcal{L}_d$, then for all $E \subset \Omega$ Borel we have
\begin{equation}
    \rhonc(E) \leq \nu(E). \label{eq:nu_dominates_rhonc}
\end{equation}
As such, if $\nu \ll \mathcal{L}_d$, we have $\rhonc \ll \mathcal{L}_d$ as well. 
\end{lemma}
\begin{proof}
Observe that if 
\begin{equation}
\gamma_0 |_{\{|x-y| \leq \step \}} = (I, T_\step)_\# \muc, \label{eq:smallscalegamma0}
\end{equation}
then we are done, since then for $E \subset \Omega$ Borel,
\begin{align*}
\rhonc(E) &= (\pi_y)_\#(I, T_\step)_\# \muc(E),\\
&= (\pi_y)_\# \gamma_0|_{\{|x-y| \leq \step \}}(E),\\
&= \gamma_0( \Omega \times E \cap \{ |x-y| \leq \step\}),\\
&\leq \gamma_0 (\Omega \times E),\\
&= \nu(E).
\end{align*}
So, we focus on proving \eqref{eq:smallscalegamma0}. Note first that if $\gamma_0|_{\{|x-y\} \leq \step}$ is the zero measure, then \eqref{eq:smallscalegamma0} automatically holds. We therefore proceed assuming that
\begin{equation*}
\gamma_0|_{\{|x-y| \leq \step\}}(\Omega \times \Omega)>0.
\end{equation*}
Recall the potential $\varphi_\step$, optimal in \eqref{prob:firstdual}. Since $\varphi^{c_2}_\step$ is Lipschitz, it is differentiable almost everywhere. Thus,  $\lebesgue(\partial \Omega)=0$ implies that there exists a Borel measurable set $G \subset \Omega \setminus \partial \Omega$ such that $\varphi_\step^{c_2}$ is differentiable on $G$, $T_\step(x) = x - \nabla \varphi^{c_2}_\step(x)$ on $G$, and $\mathcal{L}_d(G^c) = 0$. We therefore have, for $E_1, E_2 \subset \Omega$ Borel,
\begin{align*}
\gamma_0|_{\{|x-y| \leq \step\}}(E_1 \times E_2) &= \gamma_0( E_1 \times E_2 \cap \{|x-y| \leq \step\}),\\
&= \gamma_0( E_1 \times E_2 \cap \{|x-y| \leq \step\} \cap \spt(\gamma_0) \cap G \times \Omega).
\end{align*}
Here, the second equality holds since $\mu \ll \lebesgue$. Next, we claim that
\begin{equation*}
\{|x-y| \leq \step\} \cap \spt(\gamma_0) \cap G \times \Omega \subset \Gamma_{T_\step}(G),
\end{equation*} 
the latter being the graph of the map $T_\step$ over $G$. Indeed, if $(x,y)$ is in the set on the left hand side, then according to \Cref{lem:gradientofvarphi} we get $y = x- \nabla \varphi^{c_2}_0(x) = T_\step(x)$, which proves the claim. We observe, however, that
\begin{equation*}
(E_1 \times E_2) \cap \Gamma_{T_\step}(G) = (E_1 \cap T_\step^{-1}(E_2) \times \Omega) \cap \Gamma_{T_\step}(G).
\end{equation*}
As such,
\begin{align*}
\gamma_0|_{\{|x-y| \leq \step\}}(E_1 \times E_2) &= \gamma_0( E_1 \times E_2 \cap \{|x-y| \leq \step\} \cap \spt(\gamma_0) \cap G \times \Omega),\\
&= \gamma_0( E_1 \cap T_\step^{-1}(E_2) \times \Omega \cap \{|x-y| \leq \step\}),\\
&= \gamma_0|_{\{|x-y| \leq \step\}}(E_1 \cap T_\step^{-1}(E_2) \times \Omega),\\
&= \muc(E_1 \cap T_\step^{-1}(E_2)),\\
&= (I, T_\step)_\# \muc(E_1 \times E_2).
\end{align*}
Thus, $\gamma_0|_{\{|x-y| \leq \step\}}$ and $(I, T_\step)_\# \muc$ agree on all measurable rectangles $E_1\times E_2$. Since $\muc(\Omega) = \gamma_0|_{\{|x-y| \leq \step\}}(\Omega\times \Omega)$, we can multiply  $\gamma_0|_{\{|x-y| \leq \step\}}$ and $(I, T_\step)_\# \muc$ by the same constant to obtain probability measures. These probability measures agree on all measurable rectangles, and hence by Theorem 3.3 of \cite{billingsley2008probability} they are equal. This implies \eqref{eq:smallscalegamma0}, completing the proof. 
\end{proof}
\begin{remark}
We note that \Cref{lem:rho_0_bounded_above} implies that absolute continuity of $\mu$ is not enough to obtain $\rho_\step\ll \lebesgue$. Indeed, if $\nu$ and $\lebesgue$ are singular and $\muc$ is non-zero, then $\rhonc$ is non-zero and \Cref{lem:rho_0_bounded_above} implies that $\rhonc$ and $\lebesgue$ are singular. Thus, for singular $\nu$, $\rho_\step$ may have a non-zero singular component with respect to Lebesgue measure. 
\end{remark}
\subsection{Proof that $\rhond \ll \mathcal{L}_d$}
The general idea of the argument is to take $E\subset \Omega$ Borel with measure $0$ and write
\begin{align*}
\rhond(E) &= (\pi_x)_\# \gamma_0|_{|x-y| > \step} (T_\step^{-1}(E)),\\
&= \gamma_0( T_\step^{-1}(E) \times \Omega \cap \{ |x-y| > \step \} \cap \spt(\gamma_0)).
\end{align*}
We will show that the set in the preceding line is contained in a set of the form $A \times \Omega$ for $A$ Borel with measure $0$, which will guarantee that $\rhond(E) = 0$ since $\mu \ll \lebesgue$. 

We will start with some simple observations about the set of $(x,y)$ inside the support of $\gamma_0$ with $|x-y|> \step$. We will use the notion of the transport rays of a $1$-Lipschitz function (see, for example, Definition 3.7 of \cite{santambrogio2015optimal}).
\begin{lemma}
If $(x,y) \in  \spt(\gamma_0)$ with $|x-y|>\step$, then $x$ and $y$ are in transport rays of $\varphi_\step^{c_2}/\step$ and $-\varphi_\step/\step$, respectively. If $\varphi^{c_2}_\step$ is differentiable at $x$ and $x$ is an interior point of $\Omega$, then the increasing directions of both rays are parallel to $\nabla \varphi^{c_2}_\step(x)$. Further, $x-\nabla \varphi_\step^{c_2}(x)$ is in the same ray as $y$, and this is the unique transport ray of $-\varphi_\step/\step$ containing $x- \nabla \varphi^{c_2}_\lambda(x)$. \label{lem:sourceandtargetisintransportray}
\end{lemma}

\begin{proof}
Since $(x,y) \in \spt(\gamma_0)$, Kantorovich duality gives us that
\begin{equation*}
\varphi^{c_2}_\lambda(x) + \varphi_\step(y) = \hybrid(x,y).
\end{equation*}
By the equality $\varphi^{c_2}_\step = \varphi^{\hybrid}_\step$,
\begin{equation*}
\varphi^{c_2}_\lambda(x) = \inf_{z \in \Omega} \hybrid(x,z) - \varphi_\step(z),
\end{equation*}
and so we know the infimum is obtained at $y$. Note that since $|x-y| > \step$, by traversing the segment $[x,y]$ starting at $y$ we obtain a rate of decrease of $\step$ per unit distance for $\hybrid(x,\cdot)$. Since $-\varphi_\step$ is $\step$-Lipschitz, we must therefore have that the infimum is also obtained at every point $z \in [x,y]$ with $|x-z| \geq \step$. This is only possible if $-\varphi_\step$ increases at maximal rate along this non-trivial segment, and thus $[x + \step\frac{y-x}{|y-x|}, y]$, and therefore $y$, is contained in a transport ray of $-\varphi_\step/\step$. This transport ray has increasing direction parallel to $x-y$, which will be needed later.

Since $\varphi_\step^{c_2} \in \steplip$ as well, we can prove that $x$ is in a transport ray of $\varphi^{c_2}_\step/\step$ with a nearly identical argument, starting from the equality \begin{equation*}
\varphi_\step(y) = \inf_{z \in \Omega} \hybrid(y,z) - \varphi^{c_2}_\step(z).
\end{equation*}
which holds since $\varphi_\step$ is $\hybrid$-concave.

If $\varphi^{c_2}$ is differentiable at $x$, then it is clear that $\nabla \varphi^{c_2}(x)$ is parallel to the increasing direction of the transport ray of $\varphi^{c_2}/\step$ that $x$ is in. On the other hand, the increasing direction of the transport ray of $-\varphi_\step/\step$ containing $y$ is parallel to $x-y$, which is parallel to $\nabla \varphi_\step^{c_2}$ by \Cref{lem:gradientofvarphi}. By the same lemma we have
\begin{equation*}
    x + \step\frac{y-x}{|y-x|} = x - \nabla \varphi^{c_2}_\step(x),
\end{equation*}
which verifies that $x - \nabla \varphi^{c_2}_\step(x)$ is in the same transport ray of $-\varphi_\step/\step$ as $y$.


To see that the transport ray containing $x-\nabla \varphi^{c_2}_\lambda(x)$ is unique, suppose $x-\nabla \varphi^{c_2}_\lambda(x)$ is contained in two transport rays of $-\varphi_\step/\step$. As we have shown, one of these has decreasing direction parallel to $-\nabla \varphi^{c_2}_\lambda(x)$, and since $|x-y| > \step$, non-zero length in this direction. Noting that two transport rays can only collide at a point which is the upper (or lower) endpoint of both rays (see, for example,  Lemma 10 of \cite{caffarelli2002constructing}), we get that if $x-\nabla \varphi^{c_2}_\lambda(x)$ is in a second ray, it must be at the upper endpoint of that ray. Let the decreasing direction of the other ray be given by the unit vector $v$. We compute
\begin{align*}
\frac{d}{dt}|_{t=0} \hybrid(x, x-\nabla \varphi^{c_2}_\lambda(x) + tv) - \varphi_\step(x-\nabla \varphi^{c_2}_\lambda(x) + tv) 
&= \step \langle v , -\frac{\nabla \varphi^{c_2}_\step(x)}{|\nabla \varphi^{c_2}_\step(x)|}\rangle  -\step,\\
&= \step (\langle v , \frac{y-x}{|y-x|}\rangle -1),\\
&< 0,
\end{align*}
the final inequality coming from the fact that $v \neq \frac{y-x}{|y-x|}$, and both have unit norm. As such, $t \mapsto \hybrid(x, x-\nabla \varphi^{c_2}_\step(x) + tv) - \varphi_\step(x-\nabla \varphi_\step^{c_2}(x) + tv)$ is strictly decreasing for $t \in (0,\epsilon)$ for some $\epsilon$, contradicting the fact that the infimum of $\hybrid(x,y) - \varphi_\step(y)$ is obtained at $x-\nabla \varphi^{c_2}_\step(x)$. 
\end{proof}

We can now prove that the upper endpoints of the transport rays of $-\varphi_\step/\step$ correspond to upper endpoints of the transport rays of $\varphi_\step^{c_2}/\step$.
\begin{lemma}
Suppose $(x,y) \in  \spt(\gamma_0)$ with $|x-y|>\step$, and suppose $\varphi_\step^{c_2}$ is differentiable at $x$ and $x$ is an interior point of $\Omega$. If $x-\nabla \varphi_\step^{c_2}(x)$ is at the upper end point of its transport ray of $-\varphi_\step/\step$ then $x$ is at the upper end point of a transport ray of $\varphi_\step^{c_2}/\step$.\label{lem:ifatupperendoftargetupperendofsource}
\end{lemma}
\begin{proof}
From \Cref{lem:sourceandtargetisintransportray} we have that $x$ is in a transport ray of $\varphi_\step^{c_2}/\step$. Suppose it is not the upper endpoint. Then there exists $w$ on the same transport ray obtaining a strictly larger value of $\varphi_\step^{c_2}$. As such
\begin{align*}
\varphi_\step(y) &= \hybrid(x,y) - \varphi_\step^{c_2}(x),\\
&= \hybrid(x,y) - \varphi_\step^{c_2}(w) + \step |w-x|,\\
&= \hybrid(w, y) - \varphi_\step^{c_2}(w).
\end{align*}
Here the last line holds because the transport ray that $x$ is in is parallel to the segment $[x,y]$. Since $\varphi_\step^{c_2} = \varphi_\step^{\hybrid}$,
\begin{equation*}
\varphi_\step^{c_2}(w) = \inf_{z\in\Omega} \hybrid(w, z) - \varphi_\step(z),
\end{equation*}
and the infimum is obtained at $y$. Since $|w-y| > \step + |w-x|$, we obtain that all points on the segment $[w + \step \frac{y-w}{|y-w|}, y]$ are on a transport ray of $-\varphi_\step/\step$. The point $x-\nabla \varphi_\step^{c_2}(x)$ is in the interior of this ray, and thus is not the upper endpoint.
\end{proof}

We can now prove that $\rhond$ is absolutely continuous with respect to Lebesgue measure. The general argument is the following. Take $E$ Borel negligible, $x\in T_\step^{-1}(E)$, and suppose that there exists $y$ such that $(x,y) \in \spt( \gamma_0)$ with $|x-y|>\step$. Then, ignoring $x$ at the start of transport rays of $\varphi_\step^{c_2}/\step$ (which is a Borel negligible set anyway), we can show that $x = z - \nabla \varphi_\step(z)$ for $z \in E$. Since $x$ is not at the start of its transport ray, $z$ cannot be at the start of its transport ray. Away from the endpoints of transport rays the map $z \mapsto z-\nabla \varphi_\step(z)$ is Lipschitz\footnote{see the proof of Lemma 22 of \cite{caffarelli2002constructing}, or Proposition 6 of \cite{milne2022new}.}, allowing us to conclude that our set of $x$ is Lebesgue negligible. 
\begin{proposition}
The measure $\rhond$ satisfies $\rhond \ll \lebesgue$. \label{prop:rhond_AC}
\end{proposition}
\begin{proof}
Let $E \subset \Omega$ be Borel negligible. Then
\begin{align}
\rhond(E) &= (\pi_x)_\# \gamma_0|_{|x-y| > \step} (T_\step^{-1}(E)),\nonumber\\
&= \gamma_0( (T_\step^{-1}(E)\cap \mathcal{E}^c \cap G) \times \Omega \cap \{ |x-y| > \step \} \cap \spt(\gamma_0)), \label{eq:complicatedset}
\end{align}
where $\mathcal{E}^c$ is the complement of the set of ray endpoints of $\varphi_\step^{c_2}/\step$, and $G$ is as before. Both sets are Borel and have full Lebesgue measure\footnote{For a proof that $\lebesgue(\mathcal{E}) = 0$, see Lemma 25 of \cite{caffarelli2002constructing} or Lemma 3.1.8 of \cite{klartag2017needle}.}, justifying the equality \eqref{eq:complicatedset}. If $(x,y)$ is in the set appearing in \eqref{eq:complicatedset}, then by \Cref{lem:sourceandtargetisintransportray}, $x - \nabla \varphi_\step^{c_2}(x)$ is in a unique transport ray of $-\varphi_\step/\step$, and by  \Cref{lem:ifatupperendoftargetupperendofsource} $x-\nabla \varphi_\step^{c_2}(x)$ is not at the upper endpoint of that ray. Since $|x-y| > \step$, $x-\nabla \varphi_\step^{c_2}(x)$ is also not at the lower endpoint of that ray. By Lemma 3.6 of \cite{santambrogio2015optimal}, $-\varphi_\step/\step$ is differentiable at $x-\nabla\varphi_\step^{c_2}(x)$, and \Cref{lem:sourceandtargetisintransportray} implies that
\begin{equation*}
x = x- \nabla \varphi_\step^{c_2}(x) - \nabla \varphi_\step(x- \nabla \varphi_\step^{c_2}(x)) = T_\step(x) - \nabla \varphi_\step(T_\step(x)).
\end{equation*}
Thus, if $(x,y)$ is in the set appearing in \eqref{eq:complicatedset}, then $x= z - \nabla \varphi_\step(z)$ for some $z \in E$ and in the interior of a transport ray of $-\varphi_\step/\step$. As in Proposition 6 of \cite{milne2022new}, for each $j \in \{1, 2, \ldots\}$ set $A_j$ as the set of points $z$ that are on a transport ray of $-\varphi_\step/\step$ and more than distance $1/j$ from either endpoint, and recall that by Lemma 22 of \cite{caffarelli2002constructing}, $-\nabla \varphi_\step$ is a Lipschitz function on $A_j$. We therefore obtain that if $(x,y)$ is in the set appearing in \eqref{eq:complicatedset}, then
\begin{equation*}
x \in \bigcup_{j=1}^\infty (I-\nabla \varphi_\step)(E \cap A_j).
\end{equation*}
Since $E$ is Borel negligible, and $\nabla \varphi_\step$ is a Lipschitz map on $A_j$, we obtain that the set $(I-\nabla \varphi_\step)(E\cap A_j)$ is Lebesgue measurable for all $j$ and has measure $0$. By regularity of Lebesgue measure, there exists for each $j$ a Borel set $U_j$ containing $(I-\nabla \varphi_\step^{c_2})(E\cap A_j)$ with zero Lebesgue measure. As such,
\begin{align*}
\rhond(E)&\leq \sum_{j=1}^\infty \gamma_0(U_j \times \Omega) = \sum_{j=1}^\infty \mu (U_j)=0
\end{align*}
because $\mu \ll \mathcal{L}_d$. 
\end{proof}
We have therefore proven \Cref{thm:rho0AC}, and thus statement \ref{claim:rho_step_AC} of \Cref{thm:analogue_of_wavelet_shrinkage}, by proving that $\rhond$ and $\rhonc$ are absolutely continuous (\Cref{lem:rho_0_bounded_above} and \Cref{prop:rhond_AC}).

\section{Characterisation of an optimal map for the Huber cost}
\label{sec:existence_of_huber_map}
In this section we will prove statements \ref{claim:existence_of_huber_map} and \ref{claim:rho_step_obt_from_ST} of \Cref{thm:analogue_of_wavelet_shrinkage}. The essential result is the characterisation of an optimal map transporting $\nu$ to $\mu$ for the Huber cost $\hybrid$ as a composition of a Wasserstein 2 optimal map with a Wasserstein 1 optimal map. We note that the existence of an optimal map for the cost $\hybrid$ does not follow trivially from standard results in the optimal transport literature (e.g. Theorem 1.17 of \cite{santambrogio2015optimal}) since the cost $\hybrid(x,y)$ is not a strictly convex function of $|x-y|$. 

The following lemma proves that the gradient of $\varphi_\step$ is $\nu$ almost surely unchanged by applying an optimal transport map for $W_1(\nu,\rho_\step)$, and will be useful in proving the existence of an optimal transport map for the Huber cost. Throughout this section we tacitly assume the hypotheses of \Cref{thm:analogue_of_wavelet_shrinkage}.
\begin{lemma}
\label{lem:tildeS_not_identity}
Let $S_\step$ be an optimal transport map for $W_1(\nu,\rho_\step)$, which exists since $\nu \ll \lebesgue$. Let $\varphi_\step$ be a $\hybrid$-concave solution to \eqref{prob:seconddual}. Then $\nu$ almost everywhere $\nabla \varphi_\step(y)$ and $\nabla \varphi_\step(S_\step(y))$ exist. Further if $S_\step(y) \neq y$, they satisfy
\begin{equation}
    \nabla \varphi_\step(y) = \nabla \varphi_\step(S_\step(y)) = \step \frac{y -S_\step(y)}{|y -S_\step(y)|}.\label{eq:grad_varphi_parallel}
\end{equation}
\end{lemma}
\begin{proof}
The potential $\varphi_\step$ is $\hybrid$-concave, and thus $\varphi_\step \in \steplip$. Since $\nu \ll \lebesgue$, $\varphi_\step$ is therefore differentiable $\nu$ almost everywhere. Further,
\begin{align*}
    \nu(\{ y \mid \nabla \varphi_\step(\otmapnu(y)) \text{ exists}\}) &= \nu(\otmapnu^{-1}(\{ z \mid \nabla \varphi_\step(z) \text{ exists}\}),\\
    &= \rho_\step(\{ z \mid \nabla \varphi_\step(z) \text{ exists}\}),\\
    &= 1,
\end{align*}
since $\rho_\step\ll \lebesgue$ (\Cref{thm:rho0AC}). So $\nabla \varphi_\step(\otmapnu(y))$ exists $\nu$ almost everywhere as well. Since $\varphi_\step$ solves \eqref{prob:firstdual} via \Cref{lem:first_dual_and_second_are_equivalent}, we obtain via \Cref{prop:meyer-for-multiscale} that  $\varphi_\step/\step$ is a Kantorovich potential for $W_1(\nu, \rho_\step)$. Since $S_\step$ is an optimal transport map for $W_1(\nu,\rho_\step)$ we obtain that for $\nu$ almost all $y\in \Omega$,
\begin{equation*}
    \step|S_\step(y) - y| = \varphi_\step(y) - \varphi_\step(S_\step(y)).
\end{equation*}
Thus, if $S_\step(y) \neq y$, we obtain that $[y, S_\step(y)]$ is in a transport ray of $\varphi_\step/\step$. Via Lemma 3.6 of \cite{santambrogio2015optimal} we obtain \eqref{eq:grad_varphi_parallel} whenever $\varphi_\step$ is differentiable at both $y$ and $\otmapnu(y)$. 
\end{proof}
Now we can prove the existence of an optimal transport map $S_0$ from $\nu$ to $\mu$ under the cost $\hybrid$ by composing a Wasserstein 1 optimal map from $\nu$ to $\rho_\step$ with a Wasserstein 2 optimal map from $\rho_\step$ to $\mu$. Conversely, we will prove that all such optimal $S_0$ can be written in this way. The following result proves statement \ref{claim:existence_of_huber_map} of \Cref{thm:analogue_of_wavelet_shrinkage}.
\begin{lemma}
\label{lem:existence_of_otmap_for_hybrid}
Let $\varphi_\step$ be a $\hybrid$-concave solution to \eqref{prob:seconddual}. If $R_\step$ is a Borel map almost everywhere equal to $I-\nabla \varphi_\step$, and $T_\step$ is a Borel map almost everywhere equal to $I-\nabla \varphi_\step^{c_2}$, then
\begin{equation}
    R_\step\circ T_\step(x) = x\label{eq:ae-inverses}
\end{equation}
$\mu$ almost everywhere, so we write $R_\step$ as $T_\step^{-1}$. A map $S_0$ is an optimal transport map for transporting $\nu$ to $\mu$ under the cost $\hybrid$ if and only if $S_0$ can be written as $S_0 = T_\step^{-1} \circ S_\step$, where $S_\step$ is an optimal map for $W_1(\nu,\rho_\step)$.
\end{lemma}
\begin{proof}
The claim \eqref{eq:ae-inverses} is well known, and follows since $\varphi_\step$ is a Kantorovich potential for $\squaredw{\rho_\step}{\mu}$ and $T_\step$ is an optimal transport map for $\squaredw{\mu}{\rho_\step}$, and thus
\begin{equation*}
    \varphi_\step^{c_2}(x) + \varphi_\step(T_\step(x)) = \frac{1}{2}|x-T_\step(x)|^2,
\end{equation*}
$\mu$ almost everywhere. Using $\rho_\step \ll \lebesgue$, one can then easily show \eqref{eq:ae-inverses}. Let $S_0$ be given by $S_0 = T_\step^{-1} \circ S_\step$. This same equality also implies that $(T_\lambda^{-1})_\#\rho_\step = \mu$, and so $(S_0)_\# \nu = \mu$. We now wish to prove that for $\nu$ almost all $y$,
\begin{equation*}
    \hybrid(S_0(y), y) = \hybrid(S_0(y), \otmapnu(y)) + \step|y-\otmapnu(y)|.
\end{equation*}
This is clear if $\otmapnu(y) = y$. If $\otmapnu(y) \neq y$, then this equality is an immediate consequence of  \Cref{lem:tildeS_not_identity}. We therefore compute
\begin{align*}
    \int_\Omega \hybrid(S_0(y), y) d\nu(y) &= \int_\Omega \hybrid(S_0(y), \otmapnu(y))d\nu(y) + \step\int_\Omega|y-\otmapnu(y)|d\nu(y),\\
    &= \int_\Omega c_2(T_0(z), z) d\rho_\step(z) + \step W_1(\rho_\step, \nu),\\
    &= \squaredw{\mu}{\rho_\step} + \step W_1(\rho_\step, \nu),\\
    &= \cI_{\hybrid}(\mu,\nu),
\end{align*}
where the last line follows from \Cref{cor:wrof-is-hybrid}. This verifies that $S_0$ is optimal for transporting $\nu$ to $\mu$ with the pointwise cost $\hybrid$. 

Conversely, suppose $S_0$ is optimal for transporting $\nu$ to $\mu$ under this cost. If we can prove that $T_\step \circ S_0$ is optimal for $W_1(\nu, \rho_\step)$, we will be done. Clearly, $(T_\step \circ S_0)_\# \nu = \rho_\step$, and since $\varphi_\step/\step$ is a Kantorovich potential for $W_1(\nu, \rho_\step)$ via \Cref{prop:meyer-for-multiscale}, optimality of this map will be proved if we can show that
\begin{equation}
    \step |y - T_\step(S_0(y))| = \varphi_\step(y) - \varphi_\step(T_\step(S_0(y))) \label{eq:tsteps0respectstrays}
\end{equation}
$\nu$ almost everywhere. To see this, observe that $(S_0, I)_\# \nu$ is an optimal plan for \eqref{prob:underlyingOTproblem}, and $(S_0(y), y)$ is in the support of this plan for $\nu$ almost all $y$. Conditioning $\nu$ on $|S_0(y) - y| \leq \step$, we obtain $y = T_\step(S_0(y))$ with probability $1$ via \Cref{lem:gradientofvarphi}, and thus \eqref{eq:tsteps0respectstrays} holds trivially. Conditioning on $|S_0(y) - y| > \step$, we may use \Cref{lem:sourceandtargetisintransportray} to obtain that $[T_\step(S_0(y)), y]$ is in a transport ray of $\varphi_\step/\step$ $\nu$ almost surely, which proves \eqref{eq:tsteps0respectstrays} in this case. 
\end{proof}
The following result proves statement \ref{claim:rho_step_obt_from_ST} of \Cref{thm:analogue_of_wavelet_shrinkage} by demonstrating that by applying the soft thresholding operator \eqref{eq:soft_thresholding_tmap} to the map $S_0$, one recovers $S_\step$.
\begin{proposition}
\label{prop:soft_thresholding_of_transport_map}
Let $S_0 = T_\step^{-1} \circ S_\step$ be an optimal transport map from $\nu$ to $\mu$ for the cost $\hybrid$ as obtained in \Cref{lem:existence_of_otmap_for_hybrid}. Then $\nu$ almost everywhere,
\begin{equation*}
    \otmapnu(y) = y + s_\step(|S_0(y)-y|)\frac{S_0(y)-y}{|S_0(y) -y|},
\end{equation*}
where $s_\step(|S_0(y)-y|)\frac{S_0(y)-y}{|S_0(y) -y|} = 0$ if $S_0(y) = y$. 
\end{proposition}
\begin{proof}
Take $\varphi_\step$ and $T_\step^{-1}$ as in \Cref{lem:existence_of_otmap_for_hybrid}, and set
\begin{equation*}
E:= \otmapnu^{-1}(\{ z \mid T^{-1}_\step(z) = z - \nabla \varphi_\step(z)\}).    
\end{equation*}
Then $E$ has full $\nu$ measure. If $y \in E$ and $\otmapnu(y) = y$, then $S_0(y) = T^{-1}_\step(y)$. Since $\varphi_\step \in \steplip$, we obtain that
\begin{equation*}
    y + s_\step(|S_0(y)-y|)\frac{S_0(y)-y}{|S_0(y) -y|} = y = \otmapnu(y).
\end{equation*}
If $y \in E$ and $\otmapnu(y) \neq y$, then by \Cref{lem:tildeS_not_identity}, $|S_0(y)-y|>\step$ $\nu$ almost surely. Thus, 
\begin{align*}
    y + s_\step(|S_0(y)-y|)\frac{S_0(y)-y}{|S_0(y) -y|} &= y + (|S_0(y) - y| - \step))\frac{\otmapnu(y)-y}{|\otmapnu(y) -y|},\\
    &= y + |\otmapnu(y) - y|\frac{\otmapnu(y)-y}{|\otmapnu(y) -y|},\\
    &= \otmapnu(y).
\end{align*}
\end{proof}

\section{Iterative procedures involving \eqref{prob:initial-multiscale-statement}}
\label{sec:iterative_procs}
In this section we study the iterative procedures described in \Cref{sec:intro-iterative_regularization} and \Cref{sec:intro-nonlinear_Plancherel}. The main content is a proof of \Cref{prop:convergence_to_nu} and \Cref{thm:multiscale}.

\subsection{Iterative regularization}
\label{subsec:iterative_denoising}
Here we will prove our iterative regularization result \Cref{prop:convergence_to_nu}. Recall the setting;
we take $\mu,\nu \ll \mathcal{L}_d$, and $(\step_n)_{n=0}^\infty$ a sequence of positive step sizes with sum converging to $+\infty$. Set $\mu_0 := \mu$, and for $n\geq 0$ define
\begin{equation*}
    \mu_{n+1} := \argmin_{\rho \in \mathcal{P}(\Omega)} \frac{1}{2}W_2^2(\rho, \mu_{n}) + \step_{n} W_1(\rho, \nu). 
\end{equation*}
We note that $(\mu_n)_{n=1}^\infty$ is well defined given \Cref{lem:solution_to_first_dual_exists}, \Cref{prop:meyer-for-multiscale}, and \Cref{thm:rho0AC}. The first two results establish the existence of a unique solution to the minimization problem in \eqref{def:mu_n_def} when $\mu \ll \mathcal{L}_d$, and the latter guarantees that this solution will be absolutely continuous as well. 

Before we analyse the convergence of the sequence $(\mu_n)_{n=1}^\infty$, we establish a simple estimate on $W_1(\mu_n,\nu)$.
\begin{lemma}
\label{lem:W1_distance_ito_tildemud}
Let $\Omega$ be convex and compact with non-negligible interior. Take $\mu \ll \mathcal{L}_d$, and let $\rho_\step$ solve \eqref{prob:initial-multiscale-statement}. Denoting an arbitrary optimal transport plan for the cost $\hybrid$ from $\mu$ to $\nu$ as $\gamma_0$, define $\muc$ and $\mud$ as in \eqref{def:muc_and_mud}. Then
\begin{equation}
    W_1(\rho_\step, \nu ) \leq \mud(\Omega) \diam(\Omega), \label{eq:W1_dist_above_by_diam}
\end{equation}
where $\diam(\Omega) = \sup\{ |x-y| \mid x,y \in \Omega\}$.
\end{lemma}
\begin{proof}
Recall the definitions of $\rhonc$ and $\rhond$ from \eqref{def:rhonc_and_rhond}.
Note that if $\rhond(\Omega) = 0$, we obtain via \Cref{lem:rho_0_bounded_above} that $\nu = \rhonc = \rho_\step$, and thus \eqref{eq:W1_dist_above_by_diam} holds. We therefore proceed assuming that $\rhond(\Omega) \neq 0$. We have
\begin{align*}
    W_1(\rho_\step, \nu) &= \sup_{u \in \onelip} \langle u, \rho_\step - \nu \rangle,\\
    &= \sup_{u \in \onelip} \langle u, \rhonc + \rhond - \nu \rangle,\\
    &= \sup_{u \in \onelip} \langle u,\rhond - (\nu- \rhonc) \rangle
\end{align*}
Via \Cref{lem:rho_0_bounded_above} we get that $\nu - \rhonc$ is a non-negative measure. Moreover, it has the same total mass as $\rhond$. As such
\begin{align*}
    W_1(\rho_\step, \nu) &= \rhond(\Omega)  \sup_{u \in \onelip} \langle u,\frac{\rhond}{\rhond(\Omega)} - \frac{\nu- \rhonc}{(\nu - \rhonc)(\Omega)} \rangle,\\
    &= \mud(\Omega) W_1\left( \frac{\rhond}{\rhond(\Omega)}, \frac{\nu- \rhonc}{(\nu - \rhonc)(\Omega)}\right),\\
    &\leq \mud(\Omega) \diam(\Omega),
\end{align*}
as claimed. 
\end{proof}
We can now prove our convergence result for $(\mu_n)_{n=1}^\infty$, which relies on \Cref{lem:W1_distance_ito_tildemud}. 
\begin{proof}[\textit{Proof of \Cref{prop:convergence_to_nu}}]
We first establish that $W_1(\mu_n, \nu)$ is monotonically decreasing in $n$. Indeed, by definition of $\mu_n$,
\begin{equation*}
    W_1(\mu_n,\nu) \leq W_1(\mu_{n-1}, \nu) - \frac{1}{2 \step_{n-1}}W_2^2(\mu_n, \mu_{n-1}) \leq W_1(\mu_{n-1}, \nu).
\end{equation*}
Iterating the first inequality, we also obtain that
\begin{equation*}
    W_1(\mu_n, \nu) \leq W_1(\mu,\nu) - \sum_{i=0}^{n-1}\frac{1}{2\step_i}W_2^2(\mu_{i+1}, \mu_{i}). 
\end{equation*}
Thus, 
\begin{equation}
    \sum_{i=0}^{\infty}\frac{1}{2\step_i}W_2^2(\mu_{i+1}, \mu_{i})< \infty.\label{eq:sum_of_W2squared_finite}
\end{equation} 
For each $i$, let $\gamma_i$ be an optimal plan for the transport from $\mu_i$ to $\nu$ under the cost $c_{2, \step_i}$, and define
\begin{equation*}
    \mudi := (\pi_x)_\#(\gamma_i|_{|x-y|> \step_i}).
\end{equation*}
Let $\varphi_i$ be a solution to \eqref{prob:firstdual} with $\mu$ replaced by $\mu_i$ and $\lambda$ replaced by $\lambda_i$. Since $I-\nabla \varphi_i^{c_2}$ is almost everywhere equal to an optimal transport map from $\mu_i$ to $\mu_{i+1}$ (see \Cref{prop:meyer-for-multiscale}), and using \Cref{lem:gradientofvarphi}, we obtain
\begin{equation*}
    \frac{1}{2\step_i}W_2^2(\mu_i, \mu_{i+1}) \geq \frac{1}{2\step_i} \step_i^2 \mudi(\Omega) = \frac{1}{2}\step_i \mudi(\Omega). 
\end{equation*}
As such, \eqref{eq:sum_of_W2squared_finite} implies
\begin{equation}
    \sum_{i=1}^\infty \step_i \mudi(\Omega) < \infty \label{eq:sum_of_tildemudi_finite}.
\end{equation}
By \eqref{eq:unbounded_eta_sum}, we obtain that $\liminf_i\mudi(\Omega) = 0$. \Cref{lem:W1_distance_ito_tildemud} implies that
\begin{equation*}
    W_1(\mu_{i+1}, \nu) \leq \mudi(\Omega) \diam(\Omega).
\end{equation*}
Since $\liminf_i \mudi(\Omega) = 0$, we therefore obtain \begin{equation}
    \liminf_i W_1(\mu_{i}, \nu) = 0,\label{eq:liminf_is_step}
\end{equation}
as well. But $W_1(\mu_i, \nu)$ is monotonically decreasing in $i$, and so \eqref{eq:liminf_is_step} implies \eqref{eq:convergence_to_nu}. 
\end{proof}
\subsection{Multiscale transport and a non-linear energy decomposition}
\label{subsec:adding_detail}
In this section we prove \Cref{thm:multiscale}. We already have most of the necessary ingredients. Let us recall the setting of this procedure. We assume $\mu \ll \lebesgue$ and suppose $\step_0$ is given. For each $n \geq 0$, set $\step_{n+1} = \step_n/2$ and define
\begin{equation}
    \nu_{n+1} := \argmin_{\rho \in \mathcal{P}(\Omega)}\squaredw{\rho}{\mu} + \step_{n} W_1(\rho, \nu_{n}),\label{eq:multiscale_measures_again}
\end{equation}
where $\nu_0 := \nu$.
\begin{remark}
\label{remark:how_is_it_analogous?}
This procedure consists of iteratively solving \eqref{prob:general_minimization_problem}, starting with $y_0^*$ as $\nu_0$ and replacing it at each stage by $\nu_n$, as well as halving the scale parameter. If the same is done in the context of ROF, starting with $y_0^* = 0$, one obtains a sequence of functions $(w_n)_{n=1}^\infty$ which are the partial sums of the multiscale decomposition in \Cref{thm:multiscale_for_ROF}. In this light \eqref{eq:multiscale_measures_again} is analogous to \eqref{prob:multiscale}.
\end{remark}

\begin{proof}[\textit{Proof of \Cref{thm:multiscale}}]
The assumption $\mu \ll \lebesgue$, together with \Cref{lem:solution_to_first_dual_exists} and \Cref{prop:meyer-for-multiscale} guarantee for all $n$ that the argmin in \eqref{eq:multiscale_for_measures} exists and is unique. To prove statement \ref{claim:new_convergence}, we note that by \eqref{eq:displacement_bound_for_w2optimal} 
\begin{align*}
    \frac{1}{2}W_2^2(\mu,\nu_n) \leq \frac{1}{2}\step_{n-1}^2 &= 2^{-2n+1} \step_0^2,
\end{align*}
which proves \eqref{eq:convergence_of_nun}. To obtain the energy equality \eqref{eq:plancherel_for_measures}, we observe that
\begin{align*}
    \frac{1}{2}W_2^2(\mu,\nu) &= \frac{1}{2}W_2^2(\mu,\nu_1) + \frac{1}{2}W_2^2(\mu,\nu_0) - \frac{1}{2}W_2^2(\mu,\nu_1) - \step_0 W_1(\nu_0, \nu_1) \\
    &\quad + \step_0 W_1(\nu_0, \nu_1),\\
    &= \frac{1}{2}W_2^2(\mu,\nu_1) + D_{\step_{0}}(\nu_0, \nu_1) + \step_0 W_1(\nu_0, \nu_1).
\end{align*}
where in the second line we have used the equality of \eqref{prob:initial-multiscale-statement} and \eqref{prob:projection_problem}, proven in \Cref{prop:meyer-for-multiscale}. Iterating this equality, we obtain
\begin{equation*}
\frac{1}{2}W_2^2(\mu,\nu) = \frac{1}{2}W_2^2(\mu,\nu_k) + \sum_{n=0}^{k-1} D_{\step_{n}}(\nu_{n}, \nu_{n+1}) + \step_{n} W_1(\nu_n, \nu_{n+1}).  
\end{equation*}
Letting $k$ go to infinity and using \eqref{eq:convergence_of_nun}, we obtain \eqref{eq:plancherel_for_measures}.
\end{proof}

\bibliographystyle{siamplain}
\bibliography{bibliography}
\end{document}